\documentclass[11pt]{amsart}

\usepackage[usenames,dvipsnames]{xcolor}
\usepackage{amssymb}
\usepackage{amsmath}
\usepackage{amsthm}
\usepackage{amsfonts}
\usepackage{amscd}
\usepackage{graphicx}

\usepackage[colorlinks=true,
linkcolor=red,
filecolor=webbrown,
citecolor=webgreen]{hyperref}

\definecolor{webgreen}{rgb}{0,.5,0}
\definecolor{webbrown}{rgb}{.6,0,0}

\usepackage{color}
\usepackage{fullpage}
\usepackage{float, subfig}

\usepackage{graphics, tikz}
\usepackage{latexsym}
\usepackage{epsf}
\usepackage{breakurl}

\usepackage{scalerel}

\newcommand{\srb}[1]{\psi_{n,p} 
(#1)}
\newcommand{\rb}[2][1]{\psi_{n,p}^#1 
(#2)}

\newcommand\numberthis{\addtocounter{equation}{1}\tag{\theequation}}

\usepackage{stackengine,wasysym,scalerel}

\newcommand\reallywidetilde[1]{\ThisStyle{%
  \setbox0=\hbox{$\SavedStyle#1$}%
  \stackengine{-.1\LMpt}{$\SavedStyle#1$}{%
    \stretchto{\scaleto{\SavedStyle\mkern.01mu\AC}{.5150\wd0}}{.4\ht0}%
  }{O}{c}{F}{T}{S}%
}}

\DeclareMathAlphabet\euscr{U}{eus}{m}{n}

\newcommand{\exend}{\hfill $\Diamond$}

\theoremstyle{plain}
\newtheorem{theorem}{Theorem}
\newtheorem{corollary}[theorem]{Corollary}
\newtheorem{lemma}[theorem]{Lemma}
\newtheorem{proposition}[theorem]{Proposition}

\theoremstyle{definition}
\newtheorem{definition}[theorem]{Definition}
\newtheorem{example}[theorem]{Example}

\newtheorem{question}[theorem]{Question}

\theoremstyle{remark}
\newtheorem{remark}[theorem]{Remark}

\begin{document}

\title{Mixing properties and entropy bounds of 
\\a family of Pisot random substitutions}


\author{Giovanni B. Escolano}
\address{Department of Mathematics, Ateneo de Manila University 
 \newline
\hspace*{\parindent}Katipunan Avenue, 1108, Quezon City, Philippines}
\email{giovanni.escolano@obf.ateneo.edu,eprovido@ateneo.edu}

\author{Neil Ma\~nibo} 
\address{Fakult\"at f\"ur Mathematik, Universit\"at Bielefeld, \newline
\hspace*{\parindent}Postfach 100131, 33501 Bielefeld, Germany}
\email{cmanibo@math.uni-bielefeld.de }

\author{Eden Delight Miro}

\begin{abstract}
We consider a two-parameter family of random substitutions and show certain combinatorial and topological properties they satisfy. We establish that they admit recognisable words at every level. As a consequence, 
we get that the subshifts they define are not topologically mixing. 
We then show that they satisfy a weaker mixing property using a numeration system arising from a sequence of lengths of inflated words.
Moreover, we provide explicit bounds for the corresponding topological entropy in terms of the defining parameters $n$ and $p$. 
\end{abstract}

\maketitle

\section{Introduction}

Topological properties of symbolic dynamical systems  usually boil down to combinatorial features. Notions such as entropy, mixing, minimality, and ergodicity have combinatorial counterparts in terms of properties of (legal) words. In particular, for shift spaces generated by primitive substitutions on finite alphabets, some of these have been completely characterised: they have at most linear complexity, hence zero topological entropy \cite{Queffelec}, are never strongly mixing \cite{DK}, generate linearly reccurent hulls, and hence are minimal \cite{DL,Baake} and admit a unique ergodic measure  given by the word frequency measure \cite{Baake, Queffelec}. In this case, they also satisfy a dichotomy that either all the elements of $X$ are non-periodic, or all of them are periodic with respect to the shift action \cite{Baake}. 

There are of course several properties which have not been fully classified yet, like topological mixing \cite{DK}, which is intimately connected to the notion of $C$-balancedness (see \cite{Adam,BB}).
This property proved to be difficult to confirm or rule out even in the deterministic case. When the second eigenvalue  of the substitution matrix 
$\lambda_2$ satisfies $|\lambda_2|<1$, it is well known that the subshift is not topologically weakly mixing and hence is not topologically mixing, since both notions are equivalent for primitive substitutions \cite{Host}. 
The case when  $|\lambda_2|>1$ for binary alphabets was dealt with by 
Kenyon, Sadun and Solomyak, which were shown to be topologically mixing \cite{KSS}.

There are different ways to generalise classical substitutions, both in the symbolic and geometric settings, e.g., $S$-adic systems \cite{BD} and fusion rules \cite{FS}, where similar tools are available to access the properties mentioned above. In this work, we exclusively look at the generalisation where one sends a letter to a finite set of words, instead of just a single word. 
We call this a \emph{random substitution}, alluding to the idea that at each inflation step one chooses independently where to map each letter. 
An example is the random Fibonacci substitution $\psi: a\mapsto \left\{ab,ba\right\}, b\mapsto a$, introduced in \cite{GL}, where for each occurrence of $a$ one decides independently which between $ab$ and $ba$ to choose. Much less is known for spaces generated by these objects, although it has recently seen a lot of progress \cite{RustSpin,Gohlke,GS,Miro, BSS}. Under some non-degeneracy assumptions, they exhibit positive topological entropy and have infinitely many minimal components. Periodic and non-periodic points can coexist in this case, which happens for example for the random period doubling; see \cite[Ex.~40]{RustSpin}. 
They also typically admit infinitely many ergodic measures, some of which are frequency measures. 
Here, frequency measures are measures on the subshift whose values on cylinder sets defined by finite words are the almost sure limits of these word frequencies on the subshift. 
These measures arise from probability distributions one endows on the possible images of each letter, the ergodicity of which is proved in \cite{GS}. Measure-theoretic properties of the substitutions are not discussed in this contribution, whence we do not equip images of letters with probability distributions. For these aspects, we refer the reader to \cite{GS,GMRS}.

We focus here on a family $\left\{\psi_{n,p}\right\}$ of random substitutions first mentioned in \cite{Manibo} which we call the \emph{random noble Pisa substitutions}. This is a two-parameter generalisation of both the random noble means family \cite{BaakeMoll} and the Pisa family \cite{BaakeRauzy}, and for fixed $n\geqslant 2$ and $p\geqslant 1$
is given by 
\begin{equation}\label{eq: RNP definition}
\srb{a^{  }_i}=
\begin{cases}
\{a_1^{p-j}a^{  }_{i+1}a_1^j\mid 0\leqslant j\leqslant p\}, & \text{for }i\neq n,\\
\{a^{  }_1\}, &\text{otherwise}.
\end{cases}
\end{equation}

Below, we provide some reasons why this infinite family of random substitutions is interesting and why it is a natural choice for studying the robustness of other dynamical, spectral and topological properties apart from those we have considered in this work. 

\subsubsection*{\textbf{$\beta$-substitutions and simple Parry numbers}}
It is well known that any real number $\beta>1$ gives rise to a numeration system on $\mathbb{R}$. When $\beta$ is Pisot one can an associate a  substitution to $\beta$, which encapsulates the dynamics of the corresponding transformation on the interval $[0,1]$ \cite{DT,BS}. If further
$\beta$ is a simple Parry number, i.e., $1$ has a finite $\beta$ expansion $d_{\beta}(1)$, the  substitution is primitive, unimodular and irreducible. In particular, when $d_{\beta}(1)=[p\ldots p1]$ (where $p$ occurs $(n-1)$ times) the associated substitution is the deterministic noble Pisa substitution $\xi^{ }_{n,p}$ in Eq.~\eqref{eq: determ noble Pisa}. The family of random noble Pisa substitutions can be seen as the random analogues of these deterministic versions tied to these specific algebraic numbers. Moreover,  every deterministic substitution which shares the same substitution matrix with $\xi^{ }_{n,p}$ appears as a level-$1$ marginal of $\psi_{n,p}$.

\subsubsection*{\textbf{Random $p$-infinibonacci substitutions}}
For a fixed $p\geqslant 1$, one can also view the sequence of $\left\{\psi^{ }_{n,p}\right\}_{n\geqslant 0}$ as a sequence of approximants of the \emph{random $p$-infibonacci substitution}, which sends $a^{ }_i$ to the same set as above and 
$a^{ }_{\infty}$ to $\{a_1^{p-j}a^{  }_{\infty}a_1^j\mid 0\leqslant j\leqslant p\}$. These are random substitutions on the infinite (compact) alphabet $\mathbb{N}^{\ast}=\mathbb{N}\cup \left\{\infty\right\}$ and are random versions 
of the infinibonacci substitution introduced by Ferenczi in \cite{Ferenczi}; see also \cite{Cassaigne, PytheasFogg}. 
Using our bounds for the entropy, one can show that the entropies of the subshifts generated by  $\psi^{ }_{n,p}$ converge to $\log(p+1)/(p+1)$ as $n\to \infty$; see Proposition~\ref{prop: asymptotics}. It is interesting to find out whether this has an implication for the the random $p$-infinibonacci; compare \cite{Rezagholi}.
Very few general results are known for substitutions on infinite alphabets and almost nothing is known about their random versions, making this an interesting starting point for future work.

\subsubsection*{\textbf{Generalised Rauzy fractals for random substitutions}}
The  inflation multiplier $\lambda_{n,p}$ for any random noble Pisa substitution is always a Pisot unit, which makes this infinite family  amenable for the study of generalised Rauzy fractals for random substitutions. Here, the generalised fractal is no longer seen as the attractor of a standard IFS, but that of a Galton--Watson IFS \cite{Rust-talk}.

\subsubsection*{\textbf{Spectral theory}}
Non-degenerate random substitutions typically have mixed diffraction spectra. For a fixed ergodic measure, the almost sure diffraction $\widehat{\gamma}$ splits into 
$\widehat{\gamma}=\widehat{\gamma}^{ }_{|\mathbb{E}(\cdot)|^2}+\widehat{\gamma}^{ }_{V(\cdot)}$ consisting of the component associated to square of the expectation value and the variance respectively, where one deals with distribution-valued random variables on the line instead of weighted Dirac combs. 
These are studied in detail for the random noble means  substitutions $\left\{\psi^{ }_{2,p}\right\}_{p\geqslant 0}$ in \cite{BSS,Moll, Spindeler} where it was shown the first component is a pure point measure and the second is a purely absolutely continuous measure.

For general random substitutions, it is possible to apply a generalised balanced pair algorithm to $\psi^{ }_{n,p}$ to conclude that the first component is pure point. An interesting property for the family of random noble Pisa substitutions is that they all contain a marginal which is of Barge type (i.e., bijective on the first letters and constant on the final letters) and hence has pure point spectrum. This might be useful in studying the spectral type of $\widehat{\gamma}^{ }_{|\mathbb{E}(\cdot)|^2}$ \cite{GMR}.

\vspace{3mm}

 The paper is organised as follows. In Sections~\ref{sec: noble pisa} and \ref{sec: random subs}, we introduce the substitutions we are looking at and recall basic notions in symbolic dynamics.
We then investigate several properties satisfied by the set of level-$k$ inflation words in Section~\ref{sec: level k recog}, and conclude from there the existence of recognisable words. 
Here a word $u$ is level-$k$ recognisable with respect to a random substitution $\psi$ roughly means that $u$ has a unique preimage under $\psi^k$ and a unique decomposition into level-$k$ inflation words up to possibly a prefix and a suffix of $u$; see Definition~\ref{def: level k recog} below. In Theorem~\ref{thm: recog words}, we show that for an infinite set of parameters we indeed get such words for any level $k\in\mathbb{N}$. This, together with the Pisot property, implies that the subshifts they define are \emph{not} topologically mixing; see Theorem~\ref{thm: non-mixing}. 

In the same section, we also show that they nonetheless satisfy a weaker mixing property called semi-mixing, which can be proved  using a numeration system naturally arising from the lengths of inflation words of $\psi_{n,p}$; see Theorem~\ref{thm: semi-mixing}. Semi-mixing is much weaker than topological weaker but is still invariant under topological conjugacy, which can be potentially be used to classify systems which are both not topological mixing and cannot be distinguished by their topological entropy. An example would be tribonacci substitution and its twisted version given in \cite{BaakeRauzy}. These two substitutions have the same substitution matrix, and define hulls which are not mutually locally derivable (MLD). They are also not topologically conjugate, the proof of which uses  dimension arguments involving their associated Rauzy fractals. It remains to find a purely combinatorial invariant which can tell these two systems apart.

The last section is solely devoted to  topological entropy, where we apply the recent results of Gohlke in \cite{Gohlke} and estimates for $\lambda_{n,p}$ to yield bounds and asymptotic estimates for the topological entropy of the corresponding subshift. As an application, we give examples of non-conjugate random noble Pisa subshifts using the computed bounds.
We give an example of a random substitution whose topological entropy
is the logarithm of a transcendental number. This negatively answers Question 44 by Gohlke, Rust and Spindeler in \cite{GRS}. We provide a refinement of this question in Question~\ref{ques: transcendental} and provide families of examples for which this  question of the algebraic nature might be confirmed using the closed form of the corresponding entropy bounds.

\section{Noble Pisa substitutions}\label{sec: noble pisa}

Consider a finite set $\mathcal{A}=\{a_1,\ldots,a_n\}$ of letters,  which we call an alphabet. 
A finite sequence $u=u_1\ldots u_t$ over $\mathcal{A}$ is called a word, whose length is denoted by $|u|:=t$. 
By convention, if $t=0$ then $u$ is just the empty word $\epsilon$. 
The set $\mathcal{A}^*$ of all words over $\mathcal{A}$ is a monoid under concatenation i.e. for $u=u_1,\ldots,u_t$ and $v=v_1,\ldots,v_s$ one has $uv=u_1\ldots u_t v_1\ldots v_s$. 
The set $\mathcal{A}^+:=\mathcal{A}^*\setminus\{\epsilon\}$  consists of all non-empty finite words over $\mathcal{A}$. 
The full-shift $\mathcal{A}^{\mathbb{Z}}$ is the collection of all bi-infinite sequences over $\mathcal{A}$ and forms a compact metrisable space under the product topology.  
A word is a finite or infinite sequence over $\mathcal{A}$. 
We say that a word $u$ occurs in a word $v$, written as $u\prec v$, if there exists an index $\ell$ such that $u=v_{[\ell,\ell+(|u|-1)]}=v_\ell\ldots v_{\ell+(|u|-1)}$; in such case, $u$ is a subword of $v$.

A substitution is a map $\varrho\colon\mathcal{A}\to\mathcal{A}^{+}$ and extends to $\varrho\colon\mathcal{A}^\mathbb{Z}\to\mathcal{A}^\mathbb{Z}$ by concatenation, i.e., $\varrho(u)=\ldots \varrho(u_{-1}).\varrho(u_0)\varrho(u_1)\ldots$. 
This also allows one to define the powers $\varrho^k$ for $k\in\mathbb{N}$. Given $u\in\mathcal{A}^*$, we define the Abelianisation of $u$ to be the vector $\Phi(u)\in\mathbb{Z}^n_{\geqslant 0}$ with $\Phi(u)_i=|u|_{a_i}$, where $|u|_{a_i}$ denotes the number of occurrences of $a_i$ in $u$.
The language $\mathcal{L}(\varrho)$ of a substitution $\varrho$ consists of all finite words which occur in $\varrho^k(a)$ for some $k\in\mathbb{N}$ and $a\in\mathcal{A}$.
We call an element of $\mathcal{L}(\varrho)$ a legal word of $\varrho$ or $\varrho$-legal. 
The subshift $X_{\varrho}$ associated to $\varrho$ is the set of all elements in $\mathcal{A}^\mathbb{Z}$ all of whose subwords are legal. 
The (left) shift operator $S:\mathcal{A}^{\mathbb{Z}}\rightarrow\mathcal{A}^{\mathbb{Z}}$ is defined pointwise via $S(u)_i=u_{i+1}$. 
Note that the subshift $X_\varrho$ is a closed, $S$-invariant subspace of the full shift $\mathcal{A}^{\mathbb{Z}}$. 
Equipped with $S$, the subshift becomes a topological dynamical system $(X_{\varrho},S)$. 
For further information regarding substitutions and subshifts, we refer the reader to the monographs \cite{Baake,PytheasFogg,Lind}.

The family of \emph{noble means substitutions} consists of substitutions ${\sigma_p}:\{a,b\}^\mathbb{Z}\rightarrow\{a,b\}^\mathbb{Z}$ is the extension of a substitution rule $\sigma_p:\{a,b\}^*\rightarrow\{a,b\}^*,$ given by $\sigma_p(a)=a^pb,$ and $\sigma_p(b)=a$, where $p\in\mathbb{N} $. The family of \emph{Pisa substitutions} as defined by Baake and Grimm \cite{BaakeRauzy} is a set of substitutions of the form ${\varsigma_n}:\mathcal{A}^\mathbb{Z}\rightarrow\mathcal{A}^\mathbb{Z},$ induced by a substitution rule $\varsigma_n:\mathcal{A}^*\rightarrow\mathcal{A}^*,$ where $\varsigma_n(a_i)=a_1a_{i+1}$ if $i\ne n,$ and
$\varsigma_n(a_n)=a_1$, which are also called \emph{$n$-bonacci substitutions} in the literature; compare \cite{GMS,PytheasFogg}. Here, we consider the following generalisation which essentially combines the structures of these two families.

\begin{definition}
Given $n\geqslant 2$ and $p\in\mathbb{N}$, consider the alphabet $\mathcal{A}=\left\{a_1,\ldots,a_n\right\}$.
The \emph{noble Pisa substitution} ${\xi_{n,p}}\colon\mathcal{A}^\mathbb{Z}\rightarrow\mathcal{A}^\mathbb{Z}$ is induced by a substitution rule $\xi_{n,p}\colon\mathcal{A}\rightarrow\mathcal{A}^+$ given by
\begin{equation}\label{eq: determ noble Pisa}
\xi_{n,p}(a_i)=
\begin{cases}
a_1^pa_{i+1} &\text{for\ } i\ne n,\\
a_1 &\text{otherwise}.
\end{cases}
\end{equation}
\end{definition}

The Pisa substitution ${\varsigma_n}$ is the noble Pisa substitution ${\xi_{n,1}}$. For example, the Fibonacci substitution is ${\xi_{2,1}}$ while the tribonacci substitution is ${\xi_{3,1}}$. Similarly, a noble means substitution ${\sigma_p}$ corresponds to the noble Pisa substitution ${\xi_{2,p}}.$ 

The substitution matrix $M_\varrho$ of a substitution $\varrho$ encodes the occurrence of $a_i$ in $\varrho(a_j)$, i.e.,  $\left(M_\varrho\right)_{ij}=\Phi(\varrho(a_j))_i$. 
If there exists a natural number $k$ such that all the entries of $\left(M_\varrho\right)^k$ are positive, then $\varrho$ is said to be primitive. 
A substitution $\varrho$ is said to be irreducible if the characteristic polynomial $\chi_{M_\varrho}$ of $M_{\varrho}$ is irreducible in the ring $\mathbb{Q}[x]$. 
If  $\left|\det \left(M_\varrho\right)\right|=1,$ then $\varrho$ is said to be unimodular. 
A primitive substitution $\varrho$ is Pisot if the Perron--Frobenius eigenvalue $\lambda$ of $M_\varrho$ is a Pisot--Vijayaraghavan number, i.e., an algebraic integer greater than 1 while all of its algebraic conjugates are less than 1 in modulus.

\begin{proposition}\label{prop: Pisot}
Every noble Pisa substitution
${\xi_{n,p}}$ is primitive, irreducible, unimodular, and Pisot. 
\end{proposition}

\begin{proof}
For fixed $n$ and $p$, the substitution matrix of $\xi_{n,p}$ is given by  the $n\times n$ matrix 
\begin{equation}\label{eq: substitution matrix}
M_{{n,p}}
=
\begin{pmatrix}
p&p&p&\cdots&p&1\\
1&0&0&\cdots&0&0\\
0&1&0&\cdots&0&0\\
\vdots & \vdots & \vdots && \vdots&\vdots\\ 
0&0&0&\cdots&1&0\\
\end{pmatrix}. 
\end{equation}
Its characteristic polynomial $\chi_{n,p}$ is given by
$
\chi_{n,p}(x)
    =
    x^n
    -px
    \left(
    \sum\limits_{r=0}^{n-2}x^r
    \right)
    -1
$. 
Direct computation yields $\left(M_{n,p}\right)^n>\mathbf{0}$. By Brauer's result \cite[Thm.~II]{Brauer}, $\chi_{n,p}$ is an irreducible polynomial over the field $\mathbb{Q}$ and $\chi_{n,p}$ has a unique root $\lambda_{n,p}$ at the exterior of the unit circle and all other roots in the interior of the unit circle. Since $\xi_{n,p}$ is primitive, by a result from Canterini and Siegel \cite[Proposition 1.2]{Siegel}, one can conclude that ${\xi_{n,p}}$ is indeed a Pisot substitution.
Lastly, since $
   \chi_{n,p}(0)
    =-1
$, any noble Pisa substitution is unimodular. 
\end{proof}

Since $\chi_{n,p}$ is irreducible over $\mathbb{Q}$, the Perron--Frobenius eigenvalue of $M_{n,p}$ is irrational; thus, the subshift induced by $\xi_{n,p}$ is aperiodic \cite{Baake}.

\begin{remark}
The family of noble Pisa substitutions belong to the class of substitutions associated to $\beta$-numerations of real numbers, in particular, to the subclass where $\beta$ is a simple Parry number, i.e., 
where the $\beta$-expansion $d_{\beta}(1)$ of $1$ is finite; see \cite{BS,Fabre} for more details on $\beta$-substitutions. For the deterministic substitution $\xi_{n,p}$, one has $d_{\lambda_{n,p}}(1)=pp\cdots p1$, with $p$ occurring $(n-1)$ times. \exend
\end{remark}

\section{Random substitutions}\label{sec: random subs}

In this section, we define the main objects in this paper---the random noble Pisa substitutions. 
We first recall some standard definitions and notions about random substitutions; for more details, we refer the reader to the monograph \cite{Baake} and the works of Rust and Spindeler \cite{RustSpin}, and Gohlke \cite{Gohlke}. 

Let $\mathcal{P}(\mathcal{A}^+)$ the power set of $\mathcal{A}^+$. By abuse of notation, we will view a letter $a$ both as an element of $\mathcal{A}$ and of $\mathcal{P}(\mathcal{A}^{+})$ (as a singleton set).
A \emph{random substitution} $\vartheta\colon\mathcal{P}(\mathcal{A}^+)\rightarrow\mathcal{P}(\mathcal{A}^+)$ is defined inductively where a letter is no longer simply mapped to a single word
but to a set of finite words. 
Define $\vartheta(\varnothing)=\varnothing$. 
For singletons $\{w\}\in\mathcal{P}(\mathcal{A}^+),$
we have that
$
\vartheta(w)
=
\left\{
u^{(1)}\cdots u^{(|w|)}
\mid
u^{(g)}\in\vartheta(w_g)\right\}.
$
For a finite set of words $W$ one has
\begin{equation}\label{eq: image random subs}
\vartheta\left(W\right)=\bigcup\limits_{w\in W}\vartheta\left(w\right). 
\end{equation}

We call an element of $\vartheta^k(w)
$ a \emph{level-$k$ realisation} of $\vartheta$ on $w$ or a \emph{level-$k$ inflation word}. 
As in the deterministic case, $\vartheta$ extends to a map 
 $\vartheta\colon \mathcal{P}(\mathcal{A}^\mathbb{Z})\to\mathcal{P}(\mathcal{A}^\mathbb{Z})$, where for $\{y\}
\in\mathcal{P}(\mathcal{A}^\mathbb{Z})$ one has
$$\vartheta(y)
=
\left\{
x\in\mathcal{A}^\mathbb{Z}
\mid
x=\cdots x^{(-1)}.x^{(0)}x^{(1)}\cdots, \text{ where }
x^{(i)}\in\vartheta(y_i)
 \right\}.
$$
The image of a collection of bi-infinite words under $\vartheta$ can be written as in Eq.~\eqref{eq: image random subs}. We now define the random analogues of the deterministic substitutions we have in the previous section, which were first mentioned in the previous work of the second and third authors \cite{Manibo}.

A random substitution $\vartheta$ is called \emph{compatible} if for all $a\in\mathcal{A}$, all elements of $\vartheta(a)$ have the same Abelianisation, i.e., if $u,v\in \vartheta(a)$, then $\Phi(u)=\Phi(v)$. 
If $\vartheta$ is compatible, one can define the substitution matrix of $\psi$ via $\Phi$ as in the deterministic case. 
This allows one to extend the notion primitivity, irreducibility, unimodularity and the Pisot property to compatible random substitutions. A \emph{marginal} $\varrho\colon \mathcal{A}\to \mathcal{A}^{+}$ of a random substitution $\vartheta$ is a deterministic substitution which satisfies $\varrho(a)\in\vartheta(a)$ for all $a\in\mathcal{A}$.  Compatibility is equivalent to all marginals $\varrho$ of $\vartheta$ having the same substitution matrix.

One can also extend the notion of legality in the random setting. 
A word $w\in\mathcal{A}^+$ is said to be \emph{legal with respect to}  $\vartheta$ if there exist a natural number $k\in\mathbb{N}$ and a letter $a\in\mathcal{A}$ such that $w\prec u$ for some $u\in \vartheta^{k}(a)$; in such case, we write $w\in\mathcal{L}(\vartheta)$, where $\mathcal{L}(\vartheta)$ is the \emph{language of the random substitution} $\vartheta$. The set of length-$\ell$ legal words $\mathcal{L}^\ell(\vartheta)$ is the collection of $\vartheta$-legal words of length $\ell.$ 
As in the deterministic setting, we define the subshift $X_\vartheta$ of $\vartheta$ via the language, i.e., 
all bi-infinite sequences all of whose subwords are legal with respect to $\vartheta$.  

The \emph{language $\mathcal{L}(X_\vartheta)$ of the subshift} $X_\vartheta$ is the set $\{u\prec v\mid v\in X_{\vartheta}\},$ and elements $u$ of $\mathcal{L}(X_\vartheta)$ are said to be \emph{admitted} or \emph{admissible}. The set $\mathcal{L}^\ell(X_\vartheta)$ is the collection of admissible words of length $\ell$. 
The language of a random substitution is not necessarily the same as the language of the subshift. 
Nevertheless, for primitive random substitutions, legality and admissibility are equivalent.

 By construction, all random noble Pisa substitutions are compatible; hence, the deterministic and the random substitutions matrices are the same. Thus, $\psi_{n,p}$ is also primitive, irreducible, unimodular, and Pisot.
For ease of notation, we let $X_{n,p}$ be the subshift of the random noble Pisa substitution $\psi_{n,p}$.  Since $\lambda_{n,p}\notin \mathbb{N}$, $X_{n,p}$ does not contain periodic points for any pair of parameters $n,p$; see \cite[Prop.~26]{Rust}. We also have $\mathcal{L}(\psi_{n ,p})=\mathcal{L}(X_{n,p})$ for $\psi_{n,p}$; compare \cite[Lem.~9]{RustSpin}.

\section{Mixing properties}\label{sec: level k recog}
In this section, we investigate the mixing properties of the dynamical systems induced by random noble Pisa substitutions. 

\subsection{Topological mixing}
A topological dynamical system $(M,f)$ is called \emph{topologically mixing} if for every pair $(U,V)$ of  open sets there exists a natural number $N\geqslant0$ such that  $f^m(U)\cap V\neq\varnothing$ for all $m\geqslant N$.
When one considers a subshift $(X,S)$ over a finite alphabet, this topological definition is equivalent to a combinatorial condition in terms of legal words. More formally, we have the following. 


\begin{proposition}[Mixing for subshifts]\label{prop: mixing subshift}
A subshift $X$ is topologically mixing if for every ordered pair of legal words $u,v \in \mathcal{L}(X)$ there exists  a natural number $N\in\mathbb{N}$ such that if $m\geqslant N,$ there exists a legal word $w\in\mathcal{L}^m(X)$ of length $m$ such that $uwv\in\mathcal{L}(X)$, i.e., it is also a legal word. \qed
\end{proposition}

This property has been studied for several classes of deterministic substitutions \cite{DK,KSS} and, recently, for random substitutions \cite{Miro}.
A sufficient condition for a  Pisot random substitution 
subshift $(X_{\vartheta},S)$ to not be topologically mixing in terms of recognisable words is given in \cite{Miro} as follows. 

\begin{theorem}[{\cite[Thm.~35]{Miro}}]\label{thm: main}
A random substitution $\vartheta$ that is Pisot and has recognisable words at all levels induces a dynamical system $(X_{\vartheta},S)$ that is not topologically mixing. \qed
\end{theorem}

The proof of Theorem~\ref{thm: main} uses $C$-balancedness (which comes from the Pisot property) and the existence of level-$n$ recognisable words, say $u$, to demonstrate 
that consecutive occurrences of $u$ admit a hierarchy of ``missing lengths" between them, hence violating the condition in Proposition~\ref{prop: mixing subshift}. 

\begin{remark}
The notion of $C$-balancedness for substitutions was introduced by Adamczewski in \cite{Adam} as a generalisation of balancedness for binary alphabets introduced by Morse and Hedlund. When $C=1$, a $1$-balanced word $x$ is simply called balanced, which is exactly the notion used to characterise Sturmian words. 
An infinite word $x\in\mathcal{A}^{\mathbb{N}}$ over a finite alphabet is called  $C$-balanced if for every finite subword $u,v$ of $w$ with $|u|=|v|$, one has $||u|_{a}-|v|_a|\leqslant C$ for all $a\in\mathcal{A}$. One can also talk about $C_{w}$-balancedness on a specific finite word $w$, where one replaces $a$ with $w$. 
We refer the reader to \cite{BB} for a more comprehensive account on $C$-balancedness for subshifts.  
\exend
\end{remark}

In what follows, we show that the random noble Pisa substitution $\psi_{n,p}$ admits recognisable words at all levels for $\{n,p\}\subset\mathbb{N}\setminus\{1\}$; that is, for any natural number $k>0$, there exists a level-$k$ recognisable legal word; see Theorem~\ref{thm: recog words}. 
Since all random noble Pisa substitutions are Pisot, the next result follows directly from Theorem \ref{thm: main}. 

\begin{theorem}\label{thm: non-mixing}
Given $\{n,p\}\subset\mathbb{N}\setminus\{1\},$ the random noble Pisa substitution $\psi_{n,p}$ induces a dynamical system $(X_{n,p},S)$ that is not topologically mixing. \qed
\end{theorem}

The concept of recognisable words is derived from the general question of recognisability: given a substitution $\vartheta$ and a $\vartheta$-legal word $u$, can we uniquely decompose $u$ as a concatenation of inflation words $\cup_{a\in \mathcal{A}}\,\, \vartheta(a)$?
As in the deterministic setting, a random substitution $\vartheta$ is said to be globally recognisable if for each bi-infinite word $x \in X_\vartheta$ there is a unique preimage $y\in X_\vartheta$ such that $S^k(x) \in \vartheta(y)$ for some unique natural number $k$ \cite{Rust}; see \cite{BSTY} for the analogue for $S$-adic systems. 
But unlike deterministic substitutions, where aperiodicity is equivalent to global recognisability \cite{Mosse}, known examples of random substitutions like the random Fibonacci and, in general, the random noble Pisa are not recognisable; see Proposition~\ref{prop: identical set cond noble Pisa}. 
Fortunately, as stated in Theorem \ref{thm: main}, we do not need global recognisability to establish absence of mixing.

Here, we formally define the notion of level-$k$ recognisability for $\vartheta$-legal words introduced in \cite{Miro}. We first define the level-$k$ inflation word of a given legal word. 

\begin{definition}[Level-$k$ inflation word decomposition, \cite{Miro}]\label{def: level-k IWD}
Let $k\in\mathbb{N}$, $\vartheta$ a random  substitution, and $u\in \mathcal{L}(\vartheta)$ a legal word. 
A \emph{level-$k$ inflation word decomposition  of $u$} is a tuple $\left(\left[u^{(1)},\ldots,u^{(|v|)}\right], v\right)$, where $v$ is a legal word and $u=u^{(1)}\cdots u^{(|v|)}$ such that $u^{(i)} \in \vartheta^k(v_i)$ for all $2\leqslant i \leqslant |v|-1$, $u^{(1)}$ is a suffix of a realisation in $\vartheta^k(v_1)$, and $u^{(|v|)}$ is a prefix of a realisation in $\vartheta^k(v_{|v|})$. We let $\mathcal{D}_{\vartheta^k}(u)$ denote the set of all level-$k$ inflation word decompositions of $u$.

The tuple $\left[u^{(1)},\ldots,u^{(|v|)}\right]$ is called a \emph{$\vartheta^k$-cutting} of $u$, while the word $v$ is called a \emph{level-$k$ root} of $u$ associated to a $\vartheta^k$-cutting of $u$.  
\end{definition}

\begin{definition}\label{def: level k recog}
Let $\vartheta$ be a random substitution on $\mathcal{A}$.
A legal word $u\in \mathcal{L}(\vartheta)$ is said to be a \emph{level-$k$ recognisable word with respect to $\vartheta$} if the $\vartheta^k$-cutting of $u$ is unique and $u$ uniquely determines its central root $v_2\cdots v_{|v|-1}$ if $|v|> 2$; or if $|v| \leqslant 2$, $u$ uniquely determines its root $v$. 
\end{definition}

It should be noted that our definition of recognisable word is stronger than the definition given in \cite{Miro}, where a legal word $u$ is defined as level-$k$ $N$-recognisable if for all legal words of the form $w=u^{(l)}uu^{(r)}$ with $|u^{(l)}|=|u^{(r)}|=N$, all level-$k$ inflation word decompositions of $w$ induce a unique level-$k$ inflation word decomposition on $u$. In fact, if $u$ satisfies the conditions in Definition \ref{def: level k recog}, then it is level-$k$ $0$-recognisable. 

\begin{example}[Level-$2$ recognisability]
Consider  $\psi_{3,1}:\mathcal{P}\left(\{a,b,c\}^+\right)\rightarrow\mathcal{P}\left(\{a,b,c\}^+\right)$ and its square $\psi^2_{3,1}$, which are given by
$$
\begin{array}{cc}
\psi_{3,1}=
\begin{cases}
a\mapsto \{ab,ba\}&\\
b\mapsto \{ac,ca\}&\\
c\mapsto \{a\}
\end{cases}
&
\psi_{3,1}^2=
\begin{cases}
a\mapsto 
\left\{
\begin{array}{cc}
    abac,baac,baca,abca,\\
    acab,acba,caba,caab 
\end{array}
\right\}
&\\
b\mapsto \{aab,aba,baa\}&\\
c\mapsto \{ab,ba\}
\end{cases}.
\end{array}
$$

The word $abaccaba$ has a unique $\psi_{3,1}^2$-cutting $[abac,caba]$ which determines the unique root $aa$. Hence, $\mathcal{D}_{\psi_{3,1}^2}(abaccaba) = \{([abac,caba],aa)\}$, and so $abaccaba$ is a level-$2$ recognisable word. 
However, a unique cutting does not necessarily imply that one has a unique root. 
For example, although the word $bb$ admits a unique $\psi_{3,1}^2$-cutting $[b,b]$, yet its root could have easily been any pair of letters.
Conversely, a unique root is not always associated to a unique cutting. Take the word $cac$, which has $aa$ as a unique root under $\psi_{3,1}^2$ but is associated to two $\psi_{3,1}^2$-cuttings $[ca,c]$ and $[c,ac]$. 

For longer words with a unique cutting whose associated roots are of length at least $3$, level-$k$ recognisability is only up to some prefix and suffix. 
For example, the set of all level-$2$ decompositions of the word $babaccabaa$ is given by
\begin{align*}
\mathcal{D}_{\psi_{3,1}^2}(babaccabaa)= \left\{
\begin{array}{c}
([b,abac,caba,a], aaaa),
([b,abac,caba,a], aaab),\\
([b,abac,caba,a], aaac),
([b,abac,caba,a], baaa),\\
([b,abac,caba,a], baab),
([b,abac,caba,a], baac),\\
([b,abac,caba,a], caaa),
([b,abac,caba,a], caab),\\
([b,abac,caba,a], caac)
\end{array}
\right\},
\end{align*}
from which it is clear that its unique $\psi_{3,1}^2$-cutting $[b,abac,caba,a]$ determines a unique central root $aa$. 
Hence, the word $babaccabaa$ is level-$2$ recognisable with respect to $\psi_{3,1}$
\exend

\end{example}

We are now ready to discuss a specialised construction that will allow us to construct recognisable words at all levels with respect to a random noble Pisa substitution $\psi_{n,p}$. The construction uses a map $\Gamma_{n,p}\colon \mathcal{L}(\psi_{n,p})\to \mathcal{L}(\psi_{n,p})$ with left radius 1, 
where a legal word $w$ is sent to a specific element $\Gamma_{n,p}(w)\in\psi_{n,p}(w)$. For 
$w=w_1\cdots w_{L}$ we define this realisation via $\Gamma_{n,p}(w_1w_2\cdots w_L)=\Gamma_{n,p}({ }_{w_0}w_1)\Gamma_{n,p}({ }_{w_1}w_2)\cdots\Gamma_{n,p}({ }_{w_{L-1}}w_{L})$ with $w_0=\epsilon$, which is determined by
\begin{align}
    \Gamma_{n,p}({ }_{w_{j-1}}w_j)=
    \begin{cases}
    a_1,& \text{if}\; w_j=a_n,\\
    a_{i+1}a_1^p,& \text{if}\; w_j=a_i\neq a_n\;\text{and}\;w_{j-1}\neq a_n,\\
    a_1^pa_{i+1},& \text{if}\; w_j=a_i\neq a_n\;\text{and}\;w_{j-1}= a_n.
    \end{cases}
\end{align}
This means that to decide which realisation of a letter within a word $w$ one picks, one has to look at the letter to its left. One can also view this map as $\Gamma_{n,p}=f\circ\varrho $, where $\varrho$ is a deterministic substitution on the alphabet $\mathcal{A}^{(2)}$ of left-collared letters ${}_{a}b$ coming from length-$2$ legal words in $\mathcal{L}(X_{n,p})$, and $f$ is the forgetful map, which erases the collar, i.e., $f({ }_a{b})=b$, for all ${ }_a{b}\in \mathcal{L}^{2}(X_{n,p})$.   
We refer to this construction as the $\Gamma$-construction and, when no possible confusion may arise, we simply write $\Gamma$ for $\Gamma_{n,p}$. 

\begin{example}
Consider $\psi_{3,3}: a \mapsto \{aaab, aaba, abaa, baaa\}, b\mapsto \{aaac, aaca, acaa, caaa\}, c\mapsto \{a\}$, which is a random noble Pisa substitution. Applying the $\Gamma$-construction to the $\psi_{3,3}$-legal word $acbaa$, one has
\[
\Gamma(acbaa)=\Gamma(\epsilon a)\Gamma(ac)
\Gamma(cb)\Gamma(ba)\Gamma(aa)
=
baaa|
a|  
aaac|
baaa|
baaa.
\]
One can also iterate the map to define the powers $\Gamma^{k}$ for $k\in\mathbb{N}$. 
For example, the images $\Gamma^{k}(a)$ for $1\leqslant k \leqslant 3$ are given by 
$\Gamma(a)=ba^3, 
\Gamma^{2}(a)=ca^3(ba^3)^3,$ and 
$\Gamma^{3}(a)=a(a^3b) (ba^3)^2(ca^3ba^3)^3.$
\exend
\end{example}

We have the next Lemma that highlights the following useful properties of the $\Gamma$-construction that follow immediately from the definitions; hence, we omit its proof.

\begin{lemma}
Let $\psi_{n,p}$ be a random noble Pisa substitution over the alphabet $\mathcal{A}=\{a_i\mid 1\leqslant i\leqslant n\}$.
Then, for any natural number $k$, $\Gamma^k$ is well defined and $\Gamma^k(u)$ is a realisation in $\psi_{n,p}^k(u)$, for any $u\in\mathcal{A}^+$.
Moreover, for any $a_i\in\mathcal{A}$, one has that  $\reallywidetilde{\Gamma^{k}(a_i)}\in \psi^{k}_{n,p}(a_i)$, 
where $\widetilde{w}=w_{|w|} w_{|w|-1}\cdots w_1$ is the image of $w=w_1\cdots w_{|w|}$
under reflection. \qed
\end{lemma}

\begin{example}\label{ex: np22}
Consider the random noble Pisa substitution $\psi_{2,2}: a \mapsto \{aab, aba, baa\}, b\mapsto \{a\}$ and its square 
    \[
    \begin{array}{lc}
        \psi^2_{2,2}: a \mapsto
        \left\{
        \begin{array}{c}
        aaabaab,aaababa,aaabbaa,\\
        aabaaab,aabaaba,aababaa,\\
        aabbaaa,abaaaab,abaaaba,\\
        abaabaa,ababaaa,baaaaab,\\
        baaaaba,baaabaa,baabaaa
        \end{array}
        \right\}, 
        &
           b \mapsto \{aab,aba, baa\}.
    \end{array}
\]
    The first two iterates of  $\Gamma$  on the letter $a$ are given by $\Gamma(a)=baa$ and $\Gamma^2(a)=a aab baa$. 
    It can be readily verified that $\Gamma(a)$ and $\reallywidetilde{\Gamma(a)}$ are realisations in $\psi_{2,2}(a)$  while $\Gamma^2(a)$ and $\reallywidetilde{\Gamma^2(a)}$ are realisations in $\psi_{2,2}^2(a)$.
    
    Both $\reallywidetilde{\Gamma(a)}
    \Gamma(a)=aabbaa \in \psi_{2,2}(aa)$ and $\reallywidetilde{\Gamma^2(a)}
    \Gamma^2(a)=aabbaaaaaabbaa \in \psi^2_{2,2}(aa)$ are $\psi_{2,2}$-legal words, as $aa$ is a $\psi_{2,2}$-legal word. 
    Moreover, we can show that $\reallywidetilde{\Gamma(a)}
    \Gamma(a)$ and $\reallywidetilde{\Gamma^2(a)}
    \Gamma^2(a)$ are level-$1$ and level-$2$ recognisable words, respectively. 
    For $\reallywidetilde{\Gamma(a)}\Gamma(a)=aabbaa$, since no $\psi_{2,2}$-inflation word contains the word $bb$, there must be a cutting between the two $b$'s and each $b$ must come from an inflation word in $\psi_{2,2}(a)$ resulting in a unique level-1 inflation word decomposition $\mathcal{D}_{\psi_{2,2}}(aabbaa) = \{([aab,baa],aa)\}$ or, equivalently, $\mathcal{D}_{\psi_{2,2}}\left(\reallywidetilde{\Gamma(a)}\Gamma(a)\right) = \left\{\left(\left[\reallywidetilde{\Gamma(a)},\Gamma(a)\right],aa\right)\right\}$. 
    Hence, the $\psi_{2,2}$-legal word $\reallywidetilde{\Gamma(a)}
    \Gamma(a)=aabbaa$ is level-1 recognisable.

    Now, for $\reallywidetilde{\Gamma^2(a)} \Gamma^2(a)=aabbaaaaaabbaa$, we first note that the central word $aaaaaa$ is not contained in any $\psi^2_{2,2}$-inflation word and so there must be a $\psi^2_{2,2}$-cutting within $aaaaaa$. 
    It can be verified that the only possible $\psi^2_{2,2}$-cutting of $aaaaaa$ within $aabbaaaaaabbaa$ is in the middle, i.e. $[aabbaaa,aaabbaa]$ which is exactly $\left[\reallywidetilde{\Gamma^2(a)},\Gamma^2(a)\right]$. 
    Further, we note that $\Gamma^2(a)$ (and hence $\reallywidetilde{\Gamma^2(a)}$) is longer than any inflation word in $\psi^2_{2,2}(b)$, so they are not, respectively, a prefix nor a suffix of any inflation word in $\psi^2_{2,2}(b)$. 
    Conversely, no inflation word in $\psi^2_{2,2}(b)$ is a prefix of $\Gamma^2(a)$ or a suffix of $\reallywidetilde{\Gamma^2(a)}$. 
    These observations allow us to conclude that 
    $$\mathcal{D}_{\psi^2_{2,2}}\left(\reallywidetilde{\Gamma^2(a)} \Gamma^2(a)\right) = \{([aabbaaa,aaabbaa],aa)\}=\left\{\left(\left[\reallywidetilde{\Gamma(a)},\Gamma(a)\right],aa\right)\right\},$$
    which implies that $\reallywidetilde{\Gamma^2(a)} \Gamma^2(a)=aabbaaaaaabbaa$ is a level-2 recognisable word. 
     \exend
\end{example}
    
    We now prove the key idea illustrated in the previous example in the general case. Specifically, for any natural number $k$, we show in Theorem \ref{thm: recog words} that $\reallywidetilde{\Gamma^k(a)}\Gamma^k(a)$ is a level-$k$ recognisable word with a unique inflation word decomposition $\left(\left[\reallywidetilde{\Gamma^k(a)},\Gamma^k(a)\right],aa\right)$.
    To this end, we first need some technical results. 
   
\begin{proposition}\label{pro: not-pre-suf}
Let $\psi_{n,p}$ be a random noble Pisa substitution over $\mathcal{A}=\{a_i\mid 1\leqslant i\leqslant n\}$, where $\{n,p\}\subset\mathbb{N}\setminus\{1\}$. One then has the following:
\begin{enumerate}
    \item The word $\Gamma^k(a_1)$ is never shorter than any inflation word in $\psi_{n,p}^k(a_i)$ for all $i\neq 1$.
    \item If $y$ is an inflation word in $\psi_{n,p}^k(a_i)$ for any $i\neq 1$, then $y$ is not a prefix of $\Gamma^k(a_1)$ nor a suffix of $\reallywidetilde{\Gamma^k(a_1})$.
\end{enumerate}

\end{proposition}

\begin{proof}[Sketch of proof] 
The first statement follows from the $\Gamma$-construction and definition of the mapping $\psi_{n,p}$.
To prove the second statement, we use induction on $k$ relying on the assertion that, given a suffix $z$ of a word $y \in \psi^k_{n,p}(a_i)$ such that $z$ is not longer than a prefix $x$ of $\Gamma^k(a_1)$ and $z$ is not a prefix of $x$, 
one has that any realisation $u$ in $\psi_{n,p}(z)$ cannot be a prefix of $\Gamma(x)$ whenever $|u|\leqslant |\Gamma(x)|$. 
Then one checks all possible cases of letter configurations of $x$ and $z$ to show that each configuration produces a position where $u$ and $\Gamma(x)$ differ. 
Since we require that $z$ is not a prefix of $x$, there is a $1\leqslant D\leqslant |z|$ such that $x_{D}\neq z_D$. We then split it into the following cases:
\vspace{2mm}

    \noindent \emph{Case 1.} There is an $1\leqslant \ell\leqslant D-1$ such that $\Gamma\left(x_{\ell-1}x_{\ell}\right)\neq t^{(\ell)},$ where $u=t^{(1)}\cdots t^{(|z|)}$.
    Here one immediately gets that a letter differs somewhere among the sites $\Gamma\left(x_{\ell-1}x_{\ell}\right)$ and $t^{(\ell)}$ of $\Gamma(x)$ and $u$, respectively.
\vspace{2mm}

   \noindent \emph{Case 2.} For all  $1\leqslant \ell\leqslant D-1$ such that $\Gamma\left(x_{\ell-1}x_{\ell}\right)= t^{(\ell)},$ where $u=t^{(1)}\cdots t^{(|z|)}$. Here one needs to look farther into $x_D,$ the letters around it, as well for the counterparts of these letters in $z$, and prove that for all of these subcases, a mismatch always occurs. 
\vspace{2mm}

Notice that Case 1 deals with the possibility that $x$ and $z$ might have the same letter at the same position (i.e. $x_\ell=z_\ell$) and yet it might be the case that the specific realisation $t^{(\ell)}$ of $z_\ell$ used to form $u$ is different from $\Gamma(x_{\ell-1}x_{\ell})$. On the other hand, Case 2 deals with the more general and complicated case where $x_D\neq z_D$ so it is now possible that $t^{(\ell)}$ and $\Gamma(x_{\ell-1}x_{\ell})$ are not of equal lengths so comparisons are more difficult to facilitate.
For a more detailed treatment of each branching case, we refer the reader to \cite[Lemma~3.25]{Escolano}. 
\end{proof}

Working with the random substitution in Example \ref{ex: np22}, we illustrate the assertion in the sketch of proof for Proposition \ref{pro: not-pre-suf} in the following example. 

\begin{example}
Consider $z=aab=y \in \psi^2_{2,2}(b)$ and $x=aaa$ a prefix of $\Gamma^2(a)=aaabbaa$. 
We then have $\psi_{2,2}(z)=\{aab aab a,aba aab a, baa aab a, aab aba a, aba aba a, baa aba a, aab baa a,aba baa a, baa baa a\}$ and $\Gamma(x)=(baa)^3aa(aab)(baa)$. 
Note that all the elements of $\psi_{2,2}(z)$ are of length $7$ and the $7$-letter prefix of $\Gamma(x)$ is $baabaab.$ 
Also notice that the $7$-letter prefix of $\Gamma(x)$ and the word $baa baa a$
in $\psi_{2,2}(z)$ differ only on their last letters,  which shows how precarious and technical the assertion is. 
\exend
\end{example}

The first statement of Proposition \ref{pro: not-pre-suf} implies that $\Gamma^k(a)$ is not a prefix of any inflation word in $\psi^k_{n,p}(a_i)$ for $i\neq 1$ and  $\reallywidetilde{\Gamma^k(a)}$ is not a suffix of any such inflation words making the $\psi^k_{n,p}$-cutting in Figure \ref{fig: cuttings} (b) of the word $\reallywidetilde{\Gamma^k(a)}\Gamma^k(a)$ in Figure \ref{fig: cuttings} (a) not possible. 
Meanwhile, the second statement of Proposition \ref{pro: not-pre-suf} allows us to rule out the $\psi^k_{n,p}$-cutting in Figure \ref{fig: cuttings} (c) of the same word.
    
\begin{figure}[!ht]

\tikzset{every picture/.style={line width=0.75pt}} 
\centering
\begin{tikzpicture}[x=0.75pt,y=0.75pt,yscale=-0.8,xscale=0.8]

\draw    (156.02,70.8) -- (504.56,70.8) ;
\draw [shift={(504.56,70.8)}, rotate = 180] [color={rgb, 255:red, 0; green, 0; blue, 0 }  ][line width=0.75]    (0,5.59) -- (0,-5.59)   ;
\draw [shift={(330.29,70.8)}, rotate = 180] [color={rgb, 255:red, 0; green, 0; blue, 0 }  ][line width=0.75]    (0,5.59) -- (0,-5.59)   ;
\draw [shift={(156.02,70.8)}, rotate = 180] [color={rgb, 255:red, 0; green, 0; blue, 0 }  ][line width=0.75]    (0,5.59) -- (0,-5.59)   ;
\draw   (326.37,60.84) .. controls (326.39,56.17) and (324.07,53.83) .. (319.4,53.81) -- (250.72,53.47) .. controls (244.05,53.44) and (240.73,51.09) .. (240.75,46.42) .. controls (240.73,51.09) and (237.39,53.41) .. (230.72,53.38)(233.72,53.39) -- (162.03,53.04) .. controls (157.36,53.02) and (155.02,55.34) .. (155,60.01) ;
\draw   (505,61.67) .. controls (505.02,57) and (502.7,54.66) .. (498.03,54.64) -- (429.35,54.3) .. controls (422.68,54.27) and (419.36,51.92) .. (419.38,47.25) .. controls (419.36,51.92) and (416.02,54.24) .. (409.35,54.21)(412.35,54.22) -- (340.66,53.87) .. controls (335.99,53.85) and (333.65,56.17) .. (333.63,60.84) ;

\draw [line width=1.5]  [dash pattern={on 1.69pt off 2.76pt}]  (330.37,80.76) -- (330.37,103.99) ;
\draw    (534.89,111.69) -- (126.58,110.32) ;
\draw [shift={(126.58,110.32)}, rotate = 360.19] [color={rgb, 255:red, 0; green, 0; blue, 0 }  ][line width=0.75]    (0,5.59) -- (0,-5.59)   ;
\draw [shift={(330.73,111)}, rotate = 360.19] [color={rgb, 255:red, 0; green, 0; blue, 0 }  ][line width=0.75]    (0,5.59) -- (0,-5.59)   ;
\draw [shift={(534.89,111.69)}, rotate = 360.19] [color={rgb, 255:red, 0; green, 0; blue, 0 }  ][line width=0.75]    (0,5.59) -- (0,-5.59)   ;
\draw   (335.25,120.98) .. controls (335.22,125.65) and (337.53,128) .. (342.2,128.03) -- (425.57,128.65) .. controls (432.24,128.7) and (435.55,131.06) .. (435.52,135.73) .. controls (435.55,131.06) and (438.9,128.75) .. (445.57,128.8)(442.57,128.78) -- (528.95,129.43) .. controls (533.62,129.46) and (535.97,127.15) .. (536,122.48) ;
\draw   (126,119.44) .. controls (125.97,124.11) and (128.28,126.46) .. (132.95,126.5) -- (216.32,127.12) .. controls (222.99,127.17) and (226.3,129.53) .. (226.27,134.2) .. controls (226.3,129.53) and (229.65,127.22) .. (236.32,127.27)(233.32,127.25) -- (319.69,127.9) .. controls (324.36,127.93) and (326.71,125.62) .. (326.75,120.95) ;

\draw    (403.83,201.69) -- (254.77,200.32) ;
\draw [shift={(254.77,200.32)}, rotate = 360.53] [color={rgb, 255:red, 0; green, 0; blue, 0 }  ][line width=0.75]    (0,5.59) -- (0,-5.59)   ;
\draw [shift={(329.3,201)}, rotate = 360.53] [color={rgb, 255:red, 0; green, 0; blue, 0 }  ][line width=0.75]    (0,5.59) -- (0,-5.59)   ;
\draw [shift={(403.83,201.69)}, rotate = 360.53] [color={rgb, 255:red, 0; green, 0; blue, 0 }  ][line width=0.75]    (0,5.59) -- (0,-5.59)   ;
\draw   (330.95,210.98) .. controls (330.86,215.64) and (333.14,218.02) .. (337.8,218.12) -- (357.45,218.52) .. controls (364.12,218.66) and (367.4,221.06) .. (367.3,225.73) .. controls (367.4,221.06) and (370.78,218.8) .. (377.45,218.94)(374.45,218.87) -- (397.09,219.34) .. controls (401.76,219.43) and (404.14,217.15) .. (404.23,212.49) ;
\draw   (254.56,209.44) .. controls (254.46,214.11) and (256.74,216.49) .. (261.41,216.59) -- (281.06,216.99) .. controls (287.73,217.13) and (291.01,219.53) .. (290.91,224.19) .. controls (291.01,219.53) and (294.39,217.27) .. (301.06,217.4)(298.06,217.34) -- (320.7,217.81) .. controls (325.37,217.9) and (327.75,215.62) .. (327.84,210.95) ;

\draw  [dash pattern={on 0.84pt off 2.51pt}]  (213.97,200.32) -- (254.83,200.32) ;
\draw    (98,200.32) -- (212.03,200.32) ;
\draw [shift={(212.03,200.32)}, rotate = 180] [color={rgb, 255:red, 0; green, 0; blue, 0 }  ][line width=0.75]    (0,5.59) -- (0,-5.59)   ;
\draw [shift={(98,200.32)}, rotate = 180] [color={rgb, 255:red, 0; green, 0; blue, 0 }  ][line width=0.75]    (0,5.59) -- (0,-5.59)   ;
\draw  [dash pattern={on 0.84pt off 2.51pt}]  (446.87,201.98) -- (403.83,201.69) ;
\draw    (558,201.98) -- (449.76,201.98) ;
\draw [shift={(449.76,201.98)}, rotate = 360] [color={rgb, 255:red, 0; green, 0; blue, 0 }  ][line width=0.75]    (0,5.59) -- (0,-5.59)   ;
\draw [shift={(558,201.98)}, rotate = 360] [color={rgb, 255:red, 0; green, 0; blue, 0 }  ][line width=0.75]    (0,5.59) -- (0,-5.59)   ;

\draw    (402.95,297.05) -- (258.3,295.68) ;
\draw [shift={(258.3,295.68)}, rotate = 360.53999999999996] [color={rgb, 255:red, 0; green, 0; blue, 0 }  ][line width=0.75]    (0,5.59) -- (0,-5.59)   ;
\draw [shift={(402.95,297.05)}, rotate = 360.53999999999996] [color={rgb, 255:red, 0; green, 0; blue, 0 }  ][line width=0.75]    (0,5.59) -- (0,-5.59)   ;
\draw   (259.08,306.34) .. controls (258.99,311.01) and (261.28,313.38) .. (265.95,313.47) -- (320.34,314.46) .. controls (327.01,314.59) and (330.3,316.98) .. (330.21,321.65) .. controls (330.3,316.98) and (333.67,314.71) .. (340.34,314.83)(337.34,314.77) -- (394.74,315.83) .. controls (399.4,315.92) and (401.77,313.63) .. (401.86,308.96) ;
\draw  [dash pattern={on 0.84pt off 2.51pt}]  (218.7,295.68) -- (258.35,295.68) ;
\draw    (75.21,295.68) -- (216.82,295.68) ;
\draw [shift={(216.82,295.68)}, rotate = 180] [color={rgb, 255:red, 0; green, 0; blue, 0 }  ][line width=0.75]    (0,5.59) -- (0,-5.59)   ;
\draw [shift={(75.21,295.68)}, rotate = 180] [color={rgb, 255:red, 0; green, 0; blue, 0 }  ][line width=0.75]    (0,5.59) -- (0,-5.59)   ;
\draw  [dash pattern={on 0.84pt off 2.51pt}]  (444.72,297.34) -- (402.95,297.05) ;
\draw    (582.59,297.34) -- (447.54,297.34) ;
\draw [shift={(447.54,297.34)}, rotate = 360] [color={rgb, 255:red, 0; green, 0; blue, 0 }  ][line width=0.75]    (0,5.59) -- (0,-5.59)   ;
\draw [shift={(582.59,297.34)}, rotate = 360] [color={rgb, 255:red, 0; green, 0; blue, 0 }  ][line width=0.75]    (0,5.59) -- (0,-5.59)   ;
\draw [line width=1.5]  [dash pattern={on 1.69pt off 2.76pt}]  (330.37,124.74) -- (330.37,191.95) ;
\draw [line width=1.5]  [dash pattern={on 1.69pt off 2.76pt}]  (329.3,221) -- (329.3,288.21) ;

\draw (222.06,21.69) node [anchor=north west][inner sep=0.75pt]    {$\widetilde{\Gamma ^{k}( a _{1})}$};
\draw (400.69,28.09) node [anchor=north west][inner sep=0.75pt]    {$\Gamma ^{k}( a _{1})$};
\draw (194.58,135.91) node [anchor=north west][inner sep=0.75pt]    {$\psi ^{k}_{n,p}( a _{i})$};
\draw (399.69,138.4) node [anchor=north west][inner sep=0.75pt]    {$\psi ^{k}_{n,p}( a _{j})$};
\draw (331.93,228.4) node [anchor=north west][inner sep=0.75pt]    {$\psi ^{k}_{n,p}( a _{j})$};
\draw (257.37,225.91) node [anchor=north west][inner sep=0.75pt]    {$\psi ^{k}_{n,p}( a _{i})$};
\draw (14,52.4) node [anchor=north west][inner sep=0.75pt]    {$a)$};
\draw (13,112.4) node [anchor=north west][inner sep=0.75pt]    {$b)$};
\draw (11,202.4) node [anchor=north west][inner sep=0.75pt]    {$c)$};
\draw (13,302.4) node [anchor=north west][inner sep=0.75pt]    {$d)$};
\draw (569,152.4) node [anchor=north west][inner sep=0.75pt]    {$i,j\neq 1$};
\draw (569,242.4) node [anchor=north west][inner sep=0.75pt]    {$i,j\neq 1$};
\draw (298,328.4) node [anchor=north west][inner sep=0.75pt]    {$\psi ^{k}_{n,p}( a _{j})$};
\draw (531,332.4) node [anchor=north west][inner sep=0.75pt]    {$j=1,\dotsc ,n$};
\end{tikzpicture}

        \caption{Cuttings of the word $\widetilde{\Gamma^k(a)}\Gamma^k(a)$.}
    \label{fig: cuttings}
\end{figure}
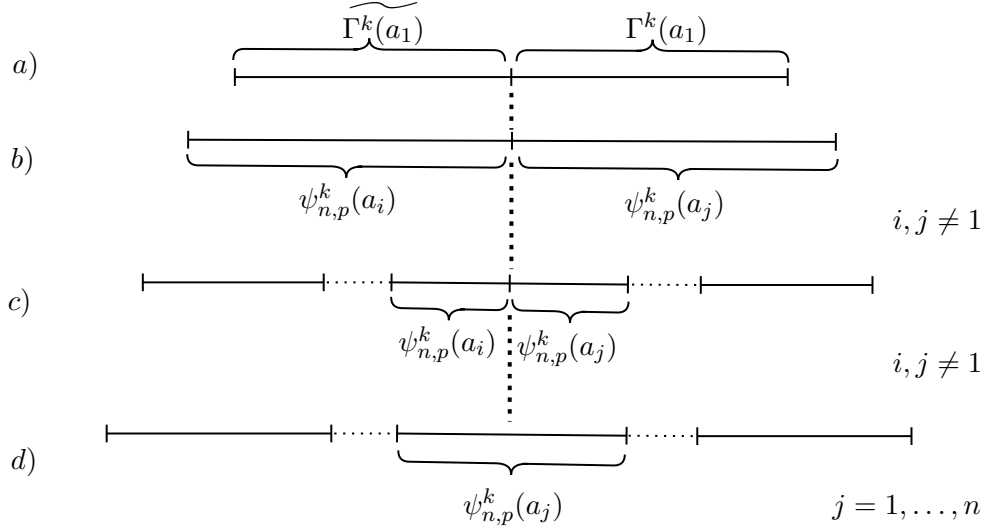 

The next result implies that the $\psi^k_{n,p}$-cutting in Figure \ref{fig: cuttings} (d) is impossible, that is, any level-$k$ inflation word cannot straddle between $\reallywidetilde{\Gamma^k(a_1)}$ and $\Gamma^k(a_1)$.
In what follows, we let $L_k$ be the length of $\Gamma^{k}(a_1)$.

\begin{lemma}\label{lem: no straddling}
Let $\psi_{n,p}$ be a random noble Pisa substitution over the alphabet $\mathcal{A}=\{a_i\mid 1\leqslant i\leqslant n\}$ for $\{n,p\}\subset\mathbb{N}\setminus\{1\}$.   
For any $1\leqslant i\leqslant n$
and any $w\in\rb[k]{a_i}$,
there are no non-empty prefix $v$ of $\Gamma^k(a_1)$ and non-empty suffix $u$ of $\reallywidetilde{\Gamma^k(a_1)}$ satisfying $w=uv$. 
\end{lemma}

\begin{proof}[Sketch of Proof]
We only prove the statement with $w\in\srb{a_i}$ by considering two cases on the first letter of $\left(\Gamma^k(a_1)\right)_1$. 
\vspace{2mm}

\noindent \emph{Case 1.} Assume $\left(\Gamma^k(a_1)\right)_1=\Gamma\left(\left(\Gamma^{k-1}(a_1)\right)_1\right)_1=a_1.$ Then, by the definition of the $\Gamma$-construction, we have that   $\left(\Gamma^{k-1}(a_1)\right)_1=a_n.$ 
Now, the fact that $\left(\Gamma^{k-1}(a_1)\right)_1=a_n$ forces $k>1$, since having $k=1$ leads to the contradiction $\left(\Gamma^{k-1}(a_1)\right)_1=(a_1)_1=a_n.$ 
We know from the definition of the $\psi_{n,p}$ mapping that $\left(\Gamma^{k-1}(a_1)\right)_{[1,2]}=a_na_1.$ 
Applying $\Gamma$ yields
\begin{align*}
\Gamma\left(\left(\Gamma^{k-1}(a_1)\right)_{[1,2]}\right)
=
a_1
a_1^pa_2
=
\left(\Gamma^k(a_1)\right)_{[1,p+2]}.\numberthis \label{eq: right1}
\end{align*}

To the contrary, suppose that  
for some $1\leqslant i\leqslant n$ there exists a $w\in\srb{a_i},$ such that for some $1 \leqslant s,t\leqslant L_k$ it is the case that
$$
w=
\reallywidetilde{\Gamma^k(a_1)}_{\left[s, L_k\right]}
\Gamma^k(a_1)_{[1,t]}
,
$$
that is, suppose $u=\reallywidetilde{\Gamma^k(a_1)}_{\left[s, L_k\right]}$ and $v=\Gamma^k(a_1)_{[1,t]}$.

Note that if $i=n$, then $w\in\srb{a_n}=\{a_1\}$ and so it follows that $w=a_1$, contradicting the assumption that $u$ and $v$ must be both non-empty words.
Hence, we now assume that $i\neq n$. 
Since $w\in\srb{a_i}=\{a_1^{p-s}a_{i+1}a_1^s\mid 0\leqslant s\leqslant p\}$, we have that for some $0\leqslant s_D\leqslant p,$ 
\begin{equation}\label{eq: which}
w=a_1^{p-s_D}a_{i+1}a_1^{s_D},
\end{equation}
and that $|w|=p+1.$
However, from Eq.~\eqref{eq: right1}, we have
\begin{equation}
        \Big.
        \reallywidetilde{\left(\Gamma^k(a_1)\right)}_{\left[L_k-(p+1), L_k\right]}
        \Big|
        \left(\Gamma^k(a_1)\right)_{[1,p+2]}
        =
        \left.
        a_2a^p_1a_1
        \right|
        a_1a^p_1a_2.
\end{equation}
Since $u=\reallywidetilde{\Gamma^k(a_1)}_{\left[s, L_k\right]}$ and $v=\Gamma^k(a_1)_{[1,t]}$ must both be non-empty, it follows that $u=a_1^{p-t+1}$ and $v=a_1^t$ implying that $w=uv=a_1^{p+1}$ which contradicts Eq.~\eqref{eq: which}. 
\vspace{2mm}

\noindent \emph{Case 2.} Assume that  $\left(\Gamma^k(a_1)\right)_1=\Gamma\left(\left(\Gamma^{k-1}(a_1)\right)_1\right)_1=a_i\neq a_1.$ Then, we have that
\begin{equation*}
    \reallywidetilde{\Gamma^k(a_1)}_{
L_k}
\Gamma^k(a_1)_{1}=a_ia_i.
\end{equation*}

Suppose, to the contrary, that  
for some $1\leqslant i\leqslant n$, there exists a $w\in\srb{a_i},$ such that for some  $1\leqslant s,t\leqslant L_k$, we have 
$$
w=
\reallywidetilde{\Gamma^k(a_1)}_{\left[s, L_k\right]}
\Gamma^k(a_1)_{[1,t]}.
$$
Then this requires $a_ia_i$ to be a subword of $w$, which is a contradiction since there is no occurrence of $a_ia_i$ in any level-1 realisation of a single letter.
The statement of the lemma with $w\in\rb[k]{a_i}$ can be proven similarly; for details, we refer the reader to \cite{Escolano}. 
\end{proof}

\begin{theorem}\label{thm: recog words}
Let $\psi_{n,p}$ be a random noble Pisa substitution over the alphabet $\mathcal{A}=\{a_i\mid 1\leqslant i\leqslant n\}$ for $\{n,p\}\subset\mathbb{N}\setminus\{1\}$. 
Then $w^{(k)}:=\reallywidetilde{\Gamma^k(a_1)}
\Gamma^k(a_1) \in \psi^k_{n,p}(a_1 a_1)$ is a level-$k$ recognisable word.
\end{theorem}

\begin{proof}

For $p>1$ and  $a_i\neq a_n$ one has $a_1a_1\prec a_1^pa_{i+1}\in\srb{a_i}$ and so $a_1a_1$ is legal.
It follows that  $w^{(k)}\in\rb[k]{a_1a_1}  
$ is legal, being an image of a legal word under $\psi^{k}_{n,p}$. Let us consider an arbitrary level-$k$ inflation word decomposition  $\left(v,\left[u^{(1)},\ldots,u^{(|v|)}\right]\right)$  of $w^{(k)}=\reallywidetilde{\Gamma^k(a_1)}\Gamma^k(a_1).$
Since $\Gamma^k(a_1) \in \psi^k_{n,p}(a_1)$, by Proposition \ref{pro: not-pre-suf} we see that $1\leqslant|u^{(1)}|\leqslant L_k< |w^{(k)}|=2L_k$.  
Thus, by Definition \ref{def: level-k IWD}, $|v|\neq1$, i.e., $w^{(k)}$ cannot be an inflation word. 

Now since $|v|>1$, Lemma \ref{lem: no straddling} guarantees that there exists a $1\leqslant j\leqslant |v|-1$ such that
\begin{align*}
    u^{(1)}\cdots u^{(j)}&=\reallywidetilde{\Gamma^k(a_1)}\\
    u^{(j+1)}\cdots u^{(|v|)}&=\Gamma^k(a_1) 
\end{align*}
with $u^{(j)}\in \psi^k(\{v_j\})$ by Definition \ref{def: level-k IWD}. 
Since $u^{(j)}$ is a suffix of $\reallywidetilde{\Gamma^k(a_1)}$, by Proposition~\ref{pro: not-pre-suf}, one has $j=1, v_1=a_1$ and $u^{(1)}=\reallywidetilde{\Gamma^k(a_1)}$. Similarly, since $u^{(2)}$ is a prefix of $\Gamma^k(a_1),$ by Proposition~\ref{pro: not-pre-suf}, $|v|=2, v_{2}=a_1$ and $u^{(2)}=\Gamma^k(a_1)$.
This proves that $\big(a_1 a_1,[\reallywidetilde{\Gamma^k(a_1)},\Gamma^k(a_1)]\big)$ is the unique 
level-$k$ decomposition of $w^{(k)}$, thus completing the proof. 
\end{proof}

\begin{remark}
It can be seen that the result above excludes the case when $p=1.$ This is because, when $p=1,$ level-$k$ recognisability is difficult to establish for words whose roots are $a_1a_n$ and $a_na_1.$ Simply consider the $\psi_{2,1}$  substitution and the level-$1$ inflation word $aba,$ whose root we cannot determine among $ab,ba,$ and $aa.$
To construct recognisable words, one needs a more specialised map than $\Gamma$, which requires one to look at letters which precede and follow a letter $a$ within a word $w$ to deduce the appropriate realisation for $a$; see \cite[Ex.~38]{Miro} for the construction for random Fibonacci $\psi_{2,1}$. \exend
\end{remark}

\subsection{Topological semi-mixing}

In Theorem \ref{thm: non-mixing}, we showed that the random noble Pisa substitution $\psi_{n,p}$ induces a dynamical system $(X_{n,p}, S)$ that is not topologically mixing, i.e.,  there exist legal words $t,w\in\mathcal{L}(X_{n,p})$ such that, for all $K\in \mathbb{N}$, there is at least one $k\geqslant K$ such that $tvw\notin\mathcal{L}(X_{n,p})$ for all $v\in \mathcal{L}^{k}(X_{n,p})$.
We now show that $(X_{n,p}, S)$ satisfies a property that is weaker than mixing called semi-mixing, which was first introduced and studied for certain types of random noble Pisa substitutions in \cite{Tadeo, Manibo}. Even though it is a weaker property, it is still preserved under topological conjugacy, and thus is a well-defined invariant; see \cite[Prop.~4.2]{Tadeo}. 

We say that a topological dynamical system $(M,f)$ is \emph{semi-mixing} with respect to a proper clopen subset $U\subset M$ if for every open set $V$ in $M$ there exists a natural number $N$ such that for every $m\geqslant N$, $f^m(V) \cap U \neq \varnothing$. Alternatively, we say that
$(M,f)$ is $U$-semi-mixing. We note here that if $(M,f)$ is topologically mixing, then $(M,f)$ is semi-mixing with respect to any clopen set $U$. 
For a subshift $(X,S)$ over a finite alphabet, semi-mixing is equivalent to the following combinatorial definition.

\begin{definition}\label{def: semi}
A subshift $X$ is $\mathcal{W}$-\emph{semi-mixing} if and only if there exists a length $\ell$ and a proper subset $\mathcal{W}\subsetneq\mathcal{L}^\ell(X)$ such that for any word $t \in \mathcal{L}(X),$ there exists a natural number $K$ such that for every $k \geqslant K$, there exists a word $v$ of length $k$ and a word $w \in \mathcal{W}$
such that $tvw \in \mathcal{L}(X)$. Here, we say that $X$ is semi-mixing with respect to $\mathcal{W}$.
\end{definition}

In \cite{Manibo}, the generalised Zeckendorf representation of natural numbers was used to show that some  well-known random noble Pisa substitutions (namely, the random Fibonacci substitution $\psi_{2,1}$, random tribonacci substitution $\psi_{3,1}$, and random metallic means substitutions $\psi_{2,p}$) are semi-mixing.  
We now show that every random noble Pisa substitution $\psi_{n,p}$ gives rise to a numeration system and utilise this numeration system to prove that all random noble Pisa substitutions are semi-mixing. Formally, we have the following result.

\begin{theorem}\label{thm: semi-mixing}
The random noble Pisa substitution $\psi_{n,p}$ induces a dynamical system $(X_{n,p},S)$ that is topologically semi-mixing with respect to $\mathcal{W}=\bigcup\limits_{i=1}^{n-1}\psi_{n,p}(a_i)\subsetneq \mathcal{L}^{p+1}(X_{n,p})$. \qed
\end{theorem}

We will need the following generalisation of the Zeckendorf and Ostrowski numerations, which is satisfied by the noble Pisa family; compare  \cite{DT,Manibo}. One can also view this as a generalisation of Brown’s criterion \cite{Brown} for complete sequences of integers.

\begin{proposition}\label{prop: suff-numeration}
Consider a sequence of non-negative integers $\{L_q\}_{q\geqslant 0}$ with $L_0=1$. 
If there exists a constant $C\in \mathbb{N}$ such that $L_q < L_{q+1} \leqslant CL_q$, then, for any natural number $N$, there exists a natural number $d>0$ and a sequence of natural numbers $0\leqslant\varepsilon_0, \ldots, \varepsilon_d \leqslant C-1$ such that 
\[
N=\sum_{q=0}^d \varepsilon_q L_q,
\]
where $\varepsilon_d \geqslant 1$. 
\qed
\end{proposition}
Note that Proposition \ref{prop: suff-numeration} only asserts the existence of a sequence $0\leqslant\varepsilon_0, \ldots, \varepsilon_d \leqslant C-1,$ so the uniqueness of the sequence is not crucial in any way.

Now, fix $n$ and $p$ and let $L_m$ be the length of any level-$m$ realisation of $\psi_{n,p}$ on $a_1$.
Note that $L_m$ does not depend on the chosen realisation because of compatibility. One can show that these lengths satisfy the linear recursion 
\[
L_{m}=pL_{m-1}+pL_{m-2}+\cdots+pL_{m-(n-1)}+L_{m-n},
\]
with initial conditions $L_m=(p+1)^m$ for $0\leqslant m\leqslant n-1$. 
The next result states that the sequence $\{L_m\}_{m \geqslant 0}$ of lengths of inflation words satisfies the conditions of Proposition~\ref{prop: suff-numeration}. 
\begin{proposition}\label{pro: L-numeration}
Let $\psi_{n,p}$ be a random noble Pisa substitution over $\mathcal{A}=\{a_i\mid 1\leqslant i\leqslant n\}.$
Then the sequence $\{L_m\}_{m \geqslant 0}$ satisfies $L_m < L_{m+1} \leqslant (p+1)L_m$.
\qed
\end{proposition}

Proposition \ref{pro: L-numeration} is easily verified from the linear recursion and initial conditions above so the proof is omitted. This result implies that given a natural number $N\in\mathbb{N}$, the set of $(n,p)$-representations of $N,$  $[N]_{n,p}=\{\varepsilon_d\cdots\varepsilon_0\mid N=\sum_{q=0}^d \varepsilon_q L_q\},$ is non-empty.

\begin{example}\label{eg: (2,2)-numeration}
Recall from Example~\ref{ex: np22} that for $\psi_{2,2}$ we have $L_0=1,$ $L_1=3,$ and $L_2=7.$ Now, $7=2(3)+1(1)=1(7)+0(3)+0(1);$ thus $\{21,100\}\subseteq [7]_{2,2}$.
\exend
\end{example}

The next result follows immediately from definitions.

\begin{lemma}\label{lem: length from set concat}
Given $
\varepsilon_d\cdots\varepsilon_0 \in [N]_{n,p}$, for any word 
$u\in
(\psi_{n,p}^d(a_1))^{\varepsilon_d}
\cdots
(\psi_{n,p}^0(a_1)^{\varepsilon_0},
$
the length of $u$ is $|u|=N$. \qed
\end{lemma}

We now note an important and useful property of this numeration---the \emph{digit-retention property}, which means that if a natural number $m > L_q$ for some $q
\in\mathbb{N}_0$, then there exists $s\in\mathbb{N}$ such that $s\geqslant q$ and $a_s\cdots a_0\in[m]_{n,p}$.
This effectively states that if $x\geqslant y$, $x$ has a representation that has at least as many digits as $y$ with respect to the numeration.
The next result could be easily gleaned from the definitions.

\begin{lemma}\label{lem: exist-q}
Given fixed $(n,p)$ and $t\in\mathcal{L}(\psi_{n,p})$, there exist a natural number $q\in\mathbb{N}_0$ and a pair of words $\{h,y\}\subseteq\mathcal{A}^*$ such that
$hty\in\psi_{n,p}^{q+2}(a_1)$. \qed
\end{lemma}

Lemma \ref{lem: exp} and Lemma \ref{lem: prefix} are simply generalisations of ideas first proposed in \cite{Manibo} regarding random noble means substitutions. These lemmas first appeared and were proven as Lemma 4.5 and Lemma 4.6, respectively, in \cite{Escolano}. Recall that $\mathcal{W}=\bigcup\limits_{i=1}^{n-1}\psi_{n,p}(a_i)$.

\begin{lemma}\label{lem: exp}
Let  $u\in\mathcal{A}^*,w\in\mathcal{W} $ and $d\in\mathbb{N}_0$ such that
$\rb[d]{a_1}uw\subseteq\mathcal{L}(\psi_{n,p})$. For any $0\leqslant j\leqslant p$, there exists $w'\in\mathcal{W}$ such that
$
\rb[d]{a_1}\srb{u}\left(\rb[0]{a_1}\right)^jw'\subseteq\mathcal{L}(\psi_{n,p})$. \qed
\end{lemma}

\begin{lemma}\label{lem: prefix}
Given $(n,p)$ and $N\in\mathbb{N}$ such that $\varepsilon_q\cdots\varepsilon_0\in[N]_{n,p}$,  for any
$$u\in
(\psi_{n,p}^{q}(a_1))^{\varepsilon_q}
\cdots
(\psi_{n,p}^{0}(a_1))^{\varepsilon_0},
$$
there exist $w\in\mathcal{W}$ and $z\in \psi^{q+2}_{n,p}(a_1)$ such that $uw$ is a prefix of $z$.  \qed
\end{lemma}

\begin{example}
We use these preceding Lemmas to illustrate the proof of Theorem \ref{thm: semi-mixing} via the example given in Example~\ref{ex: np22}. For the random noble Pisa substitution $\psi_{2,2}$,  we have the set $\mathcal{W}=\psi_{2,2}(a)=\{aab, aba,baa\}$.

Fix the $\psi_{2,2}$-legal word $t=bba$, which is a subword of $aaabbaa \in \psi_{2,2}^2(a)$. First we want to show that there exists a natural number $K$ and a legal word $v$ with $|v|=K$ such that $tvw$ is legal for some word $w \in \mathcal{W}$.

Lemma \ref{lem: exist-q} ensures the existence of $q \in \mathbb{N}_0$ and words $h,y$ such that $hty \in \psi^{q+2}_{2,2}(a)$. 
In our example, we can choose $q=0$ and $h=y=aa$ since $hty=aabbaaa\in \psi^2_{2,2}(a)$. We next describe the inductive argument to get words of increasing length between $t$ and words in the set $\mathcal{W}$. 
\vspace{2mm}

\noindent \emph{Base Case.}
To apply Lemma \ref{lem: prefix}, 
the only thing important in an $(n,p)$-representation $\varepsilon_q\cdots\varepsilon_0$  is
the value of $q$,  while $N$ can be any number that has an $(n,p)$-representation with $q+1$ digits. 
Note that in a $(2,2)$-representation in Example \ref{eg: (2,2)-numeration}, we have $L_0=1$, $L_1=3$, $L_2=7$ and so on.
Hence, for our example, since we already have $q=0$, we can arbitrarily assign $\varepsilon_0=2$ and so it follows that $N=\varepsilon_0 L_0 = 2(1)=2$ and $u=aa$ since $u\in\left(\psi^0_{2,2}(a)\right)^{\varepsilon_0}=\{a\}^{\varepsilon_0}$. 
Hence, by Lemma \ref{lem: prefix}, a pair of words $z$ and $w$ exists such that $uw$ is a prefix of $z \in \psi_{n,p}^2(a)$, and in our example we find that it is realised by $z=aabaaab\in \psi^{2}_{2,2}(a)$ with $w=baa$.

Now, we have that $htyz\in \psi^{q+2}_{n,p}(aa)$, which implies that $tyuw$ is also legal since $aa$ is legal. 
Thus, for $t=bba$, we let $v=yu$ and $K=|y|+|u|=4$.  
\vspace{2mm}

\noindent\emph{Inductive Step.}
For the inductive step, the goal is to find a word $v'$ with $|v'|=K+1=5$ such that $tv'w'$ is a legal word for some word $w'\in\mathcal{W}$.
To do this, we find a $u'$ such that $|u'|=|u|+1.$

Since $|u|$ is encoded in the numeration system with just $q+1$ digits, the digit-retention property of this numeration system guarantees that $[|u'|]_{n,p}$ has a representation that uses at least $q+1$ digits. 
This presents two possibilities:
\begin{itemize}
    \item $[|u'|]_{n,p}$ has a representation that uses precisely $q+1$ digits; or
    \item $[|u'|]_{n,p}$ has a representation that uses more than $q+1$ digits.
\end{itemize}

If  $[|u'|]_{n,p}$ has a representation $\varepsilon_q'\cdots\varepsilon_0'=[|u'|]_{n,p}$ that uses precisely $q+1$ digits, we can simply choose $u'\in
(\rb[q]{a_1})^{\varepsilon_q'}
\cdots
(\rb[0]{a_1})^{\varepsilon_0'},
$ and proceed with analogous steps when using $u$ above in the base case. 

However, if $|u'|$ has a representation that uses more than $q+1$ digits, it can be shown that $|u'|$ has a representation $\varepsilon_{q+1}''\cdots\varepsilon_0''$ that uses exactly $q+2$ digits. 
We can choose a word $x$ from the set $(\rb[q]{a_1})^{\varepsilon_{q+1}''}
\cdots(\rb[0]{a_1})^{\varepsilon_1''},$ so that
$[|x|]_{n,p}=\varepsilon_{q+1}''\cdots\varepsilon_1''$.  
Now, for any element $s\in\srb{x}$, one has  $[|s|]_{n,p}=\varepsilon_{q+1}''\cdots\varepsilon_1''0$.
Simply put,
applying the random substitution to a word induces a map to  the $(n,p)$-numeration of the length of its image by shifting all digits to the left and adding a zero at the first position.
The word $u'$ can then be selected from the set $\srb{x}\left(\rb[0]{a_1}\right)^{\varepsilon_0''}$ so that $ [|u'|]_{n,p}=\varepsilon_{q+1}''\cdots\varepsilon_1''\varepsilon_0''$, which is just the correct length we want. 

In our specific example,  we have $|u'|=|u|+1=3$ and so $[|u'|]_{2,2}=10,$ which is composed of $q+2=2$ numeration digits, namely: $\varepsilon_1''=1$ and $\varepsilon_0''=0$.
We then choose a word $x$ in $(\rb[0]{a})^{1}=\{a\}^1,$
and so we must have $x=a$ and $[|x|]_{2,2}=\varepsilon_1''=1$. 
Now we choose $u'$ from the set $\left(\psi_{2,2}(x)\right)^{\varepsilon_1''}
\left(\psi_{2,2}^0(a)\right)^{\varepsilon_0''}=\left(\psi_{2,2}(a)\right)\{\epsilon\}.$ In particular, we can pick $u'=aab$ without any prejudice against any of the elements of $\psi_{2,2}(x).$ 
We see that $u'=aab$ is of length $3$ as desired. 

By Lemma~\ref{lem: prefix}, there exists a $w\in\mathcal{W}$ and a $z'\in \psi^{q+2}_{2,2}(a)$ such that $xw'$ is a prefix of $z$. Here, $w^{\prime}=aba$ and $z'=aabaaab$.
Since $htyxw'$ is a prefix of $htyz'\in \psi^{q+2}_{2,2}(aa)$, the word $htyxw'$ is legal. 
By Lemma \ref{lem: exp}, since $htyxw'$ is legal, then there exists a $w''\in\mathcal{W}$ such that $htyu'w''$ is also legal, and so ultimately $tyu'w''$ is legal as well.
One checks that $w^{\prime\prime}=baa$ satisfies this. 
The goal for the inductive step is achieved in assigning the word $v'=yu'$  with $|v'|=|y|+|u'|=5=K+1$ such that $tv'w''$ is a legal word.
\exend
\end{example}

We now formalise the arguments above to prove Theorem~\ref{thm: semi-mixing}, i.e., that $(X_{n,p},S)$ is topologically semi-mixing with respect to $\mathcal{W}=\bigcup\limits_{i=1}^{n-1}\psi_{n,p}(\{a_i\})$. Note that $\mathcal{W}$ is a proper subset of $\mathcal{L}^{p+1}(\psi_{n,p})$ since
$a_2a_2$ is always legal, and hence extendable to a legal word $w$ of length $p+1$. This new word $w$ is clearly not in $\mathcal{W}$ since the letter $a_2$ appears at most once for any given $w^{\prime}\in\mathcal{W}$
by definition. 


\begin{proof}[Proof of Theorem \textnormal{\ref{thm: semi-mixing}}]
Following Definition~\ref{def: semi}, we need to show that, for every legal word 
$T$, there exists a natural number  $N\in\mathbb{N}$ such that, for any $m\geqslant N$, there exists a word $v\in\mathcal{A}^*$ with $|v|=m$ and a word $w\in\mathcal{W}$ such that $Tvw$ is legal. 

Given a legal word $T$, Lemma \ref{lem: exist-q} implies that we can choose the least $q\in\mathbb{N}_0$ and words  $J,Y \in \mathcal{A}^*$ such that $JTY\in \psi_{n ,p}^{q+2}(a_1).$ 
Now, consider the number $N=|Y|+L_q$, where as before $L_q=\left|\Gamma^q(a_1)\right|$.

For any natural number $m\geqslant N$, note that $m-|Y|\geqslant N-|Y|=L_q\geqslant 1.$
By the digit-retention property, for $s\geqslant q,$ we have that $a_s\cdots a_0\in[m-|Y|]_{n,p}.$
Consider  $R:=\sum\limits_{b=0}^{q}\varepsilon_{q-b}L_{q-b}$, where $\varepsilon_{q-b}:=a_{s-b}$.
Clearly, $\varepsilon_q\cdots \varepsilon_0\in[R]_{n,p}.$
 
By Lemma \ref{lem: prefix}, for any
$u\in
(\rb[q]{a_1})^{\varepsilon_q}
\cdots
(\rb[0]{a_1})^{\varepsilon_0},
$
there exists a $w\in\mathcal{W}$ and a $z\in \psi_{n,p}^{q+2}(a_1)$ such that $uw$ is a prefix of $z.$ Thus, there exists a $y\in\mathcal{A}^*,$ such that $uwy=z\in \psi_{n,p}^{q+2}(a_1).$ 
Since $a_1a_1$ is a legal word, one obtains
$
\psi_{n,p}^{q+2}(a_1)uwy
\subseteq
\psi_{n,p}^{q+2}(a_1)\psi_{n,p}^{q+2}(a_1)
=\psi^{q+2}_{n,p}(a_1^2)
\subseteq
\mathcal{L}(\psi_{n,p})
$. 
Therefore, for any $B\in \psi_{n,p}^{q+2}(a_1)$, the word $Buwy$ is legal and so $Buw$ is legal as well. 

We now consider two cases on the number of digits in an $(n,p)$-representation of the number $m-|Y|$: either exactly $q+1$ digits (i.e., $s=q$) or more than $q+1$ digits (i.e. $s>q$).
\vspace{2mm}

\noindent\emph{Case 1.} Suppose that $s=q$.
By Proposition \ref{lem: length from set concat}, since
$u\in
(\rb[q]{a_1})^{\varepsilon_q}
\cdots
(\rb[0]{a_1})^{\varepsilon_0}
$
and
$\varepsilon_q\cdots \varepsilon_0\in[R]_{n,p},$ we have that $|u|=R.$ 
By assumption, it follows that $a_s\cdots a_0=\varepsilon_q\cdots\varepsilon_0$ so that  
\begin{align*}
  R:=\sum\limits_{b=0}^{q}\varepsilon_{q-b}L_{q-b}
  =\sum\limits_{b=0}^{s}a_{s-b}L_{s-b}
  =m-|Y|.
\end{align*}
Therefore, $|u|=R=m-|Y|.$

Since we have that $JTY\in \psi_{n,p}^{q+2}(a_1)$ and $\psi^{q+2}_{n,p}(a_1)uw\subseteq\mathcal{L}(\psi_{n,p})$, it follows that the word $JTYuw\in\mathcal{L}(\psi_{n,p})$ is legal and the word $TYuw\in\mathcal{L}(\psi_{n,p})$ is also legal with $|Yu|=|Y|+|u|=m$ and $m \geqslant N$ is arbitrary, as required. 
\vspace{2mm}

\noindent \emph{Case 2.} Now, suppose that $s>q$.
Define $u^{(1)}:=\{u\}$ and for $1\leqslant x\leqslant s-q,$ let $j_x=a_{s-q-x},$ and $u^{(x+1)}=\psi_{n,p}(u^{(x)})\left(\rb[0]{a_1}\right)^{j_x}.$
Since $\psi^{q+2}_{n,p}(a_1)uw\subseteq\mathcal{L}(\psi_{n,p})$ and by Lemma \ref{lem: exp}, there exists $w^{(1)}\in\mathcal{W},$ such that
$$\psi^{q+2}_{n,p}(a_1)
\psi_{n,p}(u^{(1)})
\left(
\rb[0]{a_1}
\right)^{j_1}
w^{(1)}
\subseteq
\mathcal{L}(\psi_{n,p}).$$
It then follows that, if $2\leqslant s-q$, Lemma \ref{lem: exp} ensures that there exists $w^{(2)}\in\mathcal{W}$ such that

$$\psi^{q+2}_{n,p}(a_1)
\psi_{n,p}{(u^{(2)})}
\left(
\rb[0]{a_1}
\right)^{j_2}
w^{(2)}
\subseteq
\mathcal{L}(\psi_{n,p}).$$ 

More generally, for $1\leqslant r \leqslant s-q,$ it can be said that 
there exists a $w^{(r)}\in\mathcal{W}$ such that
$$\psi^{q+2}_{n,p}(a_1)
\psi_{n,p}{(u^{(r)})}
\left(
\rb[0]{a_1}
\right)^{j_r}
w^{(r)}
\subseteq
\mathcal{L}(\psi_{n,p}),$$ 
and so, for arbitrary
$
v^{\prime}\in
\psi_{n,p}{(u^{(r)})}
\left(
\rb[0]{a_1}
\right)^{j_r},
$
one has
$$
v^{\prime}\in
\left(
\psi_{n,p}^{q+r}(a_1)
\right)^{\varepsilon_q}
\cdots
\left(
\psi_{n,p}^{r}(a_1)
\right)^{\varepsilon_0}
\left(
\psi_{n,p}^{r-1}(a_1)
\right)^{j_1}
\cdots
\left(
\psi_{n,p}^{0}(a_1)
\right)^{j_r}.
$$
After the $(s-q)$-th iteration, we obtain 
\begin{equation}\label{eq: Esq}
\psi^{q+2}_{n,p}(a_1)
\psi_{n,p}{(u^{(s-q)})}
\left(
\rb[0]{a_1}
\right)^{j_{s-q}}
w^{(s-q)}
\subseteq
\mathcal{L}(\psi_{n,p}). 
\end{equation}
Hence, for arbitrary
$
v^{\prime}\in
\psi_{n,p}{(u^{(s-q)})}
\left(
\rb[0]{a_1}
\right)^{j_{s-q}},
$
we then have
\begin{equation*}
v^{\prime}\in
\left(
\psi_{n,p}^{s}(a_1)
\right)^{\varepsilon_q}
\cdots
\left(
\psi_{n,p}^{s-q}(a_1)
\right)^{\varepsilon_0}
\left(
\psi_{n,p}^{s-q-1}(a_1)
\right)^{j_1}
\cdots
\left(
\psi_{n,p}^{0}(a_1)
\right)^{j_{s-q}}.    
\end{equation*}
Since we let $\varepsilon_{q-b}:=a_{s-b}$ and $j_r=a_{s-q-r}$
for $1\leqslant r\leqslant s-q,$ it follows that
\begin{equation*}
v^{\prime}\in
\left(
\psi_{n,p}^{s}(a_1)
\right)^{a_s}
\cdots
\left(
\psi_{n,p}^{s-q}(a_1)
\right)^{a_{s-q}}
\left(
\psi_{n,p}^{s-q-1}(a_1)
\right)^{a_{s-q-1}}
\cdots
\left(
\psi_{n,p}^{0}(a_1)
\right)^{a_0},    
\end{equation*}
and as $a_s\cdots a_0\in[m-|Y|]_{n,p},$ we have $|v^{\prime}|=m-|Y|$ by Proposition \ref{lem: length from set concat}.

Now, since
$JTY\in \psi_{n,p}^{q+2}(a_1)$
and
$
v^{\prime}\in
\psi_{n,p}{(u^{(s-q)})}
\left(
\rb[0]{a_1}
\right)^{j_{s-q}},
$
we have that
$
JTYv^{\prime} w^{(s-q)}
$
is a legal word by Eq.~\eqref{eq: Esq} with $m=|Y|+|v^{\prime}|$ and $w^{(s-q)}\in\mathcal{W}$, proving the result for Case 2.

In both cases, we have shown that, given any legal word $T$, we can find a natural number $N=|Y|+L_q$ such that for arbitrary $m\geqslant N,$ the word $Tvw$ is legal for some word $v\in\mathcal{A}^+$ with $|v|=m$ and a word $w\in\mathcal{W}$, thus  completing the proof.
\end{proof}

\section{Entropy}
Let $X$ be a subshift over a finite alphabet defined via the language $\mathcal{L}$. The \emph{topological entropy} $h_{\text{top}}(X)$ of the symbolic dynamical system $(X,S)$ is given by 
\[
h_{\textnormal{top}}(X)=\lim_{m\to\infty} \log\frac{\mathfrak{p}(m)}{m},
\]
where $\mathfrak{p}\colon \mathbb{N}\to\mathbb{N}$ is the complexity function, i.e,  $\mathfrak{p}(m)=\#\mathcal{L}^{m}(X)$ is the cardinality of the set of all length-$m$ legal words; compare \cite{AKM,Lind,GR}.
For the random Fibonacci substitution $\psi_{1,1}$, Godr\`{e}che and Luck \cite{GL} provided a sketch suggesting that $h_{\textnormal{top}}(X_{1,1})=\sum_{m\geqslant 2} \frac{\log(m)}{\tau^{m+2}}$, where $\tau$ is the golden ratio, which was confirmed by Nilsson in \cite{Nilsson}. 
Explicit computations for $h_{\textnormal{top}}(X)$ for random noble means substitutions $\psi_{2,p}$ with $p\geqslant 1$ are found in \cite{Nilsson,BaakeMoll,Moll}. 
Recently, Gohlke showed  \cite{Gohlke} that one can get 
explicit bounds for $h_{\textnormal{top}}(X)$ for primitive semi-compatible random substitutions.

\begin{theorem}[{\cite[Thm.~15]{Gohlke}}]\label{thm: estimate top entropy}
Let $\psi$ be a primitive semi-compatible random substitution on a finite alphabet $\mathcal{A}=\left\{a^{ }_1,a^{ }_2,\ldots,a^{ }_n\right\}$. Then, one has 
\begin{equation}\label{eq: entropy bound gen}
\frac{1}{\lambda^m}\boldsymbol{q}^{T}_m\boldsymbol{R}\leqslant h_{\textnormal{top}}(X_{\psi}) \leqslant \frac{1}{\lambda^m-1}\boldsymbol{q}^{T}_m\boldsymbol{R},
\end{equation}
for all $m\geqslant 1$, where $\boldsymbol{R}$ is the normalised right Perron--Frobenius eigenvector of the substitution matrix $M_{\psi}$, $\lambda$ is the PF eigenvalue and $(\boldsymbol{q}_m)_i=\log\big( \#\big(\psi^{m}(a^{ }_i)\big)$. \qed
\end{theorem}

\begin{definition}[{\cite[Def.~7]{Gohlke}}]
A random substitution $\psi$ satisfies the \emph{identical set condition} if 
\[
u,v\in\psi(a^{ }_i)\implies \psi^m(u)=\psi^m(v)
\] 
for all $a^{ }_i\in \mathcal{A}$ and $m\in\mathbb{N}$. On the other hand, it is said to satisfy the \emph{disjoint set condition} if
\[
u,v\in\psi(a^{ }_i),u\neq v\implies \psi^m(u)\cap\psi^m(v)=\varnothing.
\]
for all $a^{ }_i\in \mathcal{A}$ and $m\in\mathbb{N}$.
\end{definition}

\begin{proposition}[{\cite[Cor.~16]{Gohlke}}]
If $\psi$ satisfies the identical set (the disjoint set) condition, the lower bound (the upper) in Eq.~\eqref{eq: entropy bound gen} is indeed an equality.  In all other cases, the inequalities are strict. \qed
\end{proposition}

In what follows, we consider the shift $X_{n,p}$ generated by the random noble Pisa substitution $\psi_{n,p}$, where $n\geqslant 2,p\in \mathbb{N}$. We denote by $h_{\textnormal{top}}(X_{n,p})$ the associated topological entropy.

\begin{proposition}\label{prop: identical set cond noble Pisa}
For any $n,p$ and $m=1$, the random substitution $\psi_{n,p}$ satisfies neither the identical set condition nor the disjoint set condition. In particular, the entropy bounds in Eq.~\eqref{eq: entropy bound gen} for $m=1$ are strict. 
\end{proposition}
\begin{proof}
Fix $n$ and $p$. To show the first claim, let $u=a^{p}_1a^{ }_n,v=a^{ }_{n}a^{ p}_1\in\psi_{n,p}(a^{ }_{n-1})$. Since $\psi_{n,p}(a^{ }_n)=a^{ }_1$, any image in $\psi_{n,p}(a^{p}_1a^{ }_n)$ must end with $a^{ }_1$. However, there are elements of the set $\psi_{n,p}(v)=\psi_{n,p}(a^{ }_{n}a^{p}_1)$ which end in $a^{ }_2$, which implies $\psi_{n,p}(u)\neq
\psi_{n,p}(v)$. For the second claim, consider $u=a^{ }_1 a^{ }_n a^{p-1}_1,v=a^{ }_n a^{p}_1\in \psi_{n,p}(a^{ }_{n-1})$. Note that the substituted word $a^{p}_{1}a^{ }_2a^{ }_1$ can either come from $a^{ }_1a^{ }_n$ or $a^{ }_na^{ }_1$. If one then maps the remaining $(p-1)$ $a^{ }_1$s in $u$ and $v$, respectively, to the same image, one finds a nontrivial word in $\psi_{n,p}(u)\cap \psi_{n,p}(v)$, which proves the assertion. 
\end{proof}

\begin{remark}
If the disjoint set condition is violated at $m=1$, it also does not hold for all $m\in\mathbb{N}$, which means the upper bound in Eq.~\eqref{eq: entropy bound gen} is strict for all $m$. The proof for this part of the previous proposition relies on the existence of words which are \emph{not} recognisable, in contrast to the words we have constructed in Section~\ref{sec: level k recog}. 
 Proposition~\ref{prop: identical set cond noble Pisa} also proves \emph{inter-alia} that $X_{n,p}$ is not globally recognisable for all $n$ and $p$.  \exend
\end{remark}

Recall that the characteristic polynomial of $M_{n,p}$ is $\chi^{ }_{n,p}(x)=x^{n}-p(x+\cdots+x^{n-1})-1$. Let $\lambda_{n,p}$ be the Perron--Frobenius (PF) eigenvalue of $M_{n,p}$, which is the unique real root of $\chi_{n,p}(x)$ satisfying $\lambda>1$.
One can explicitly check that the vector 
$(\lambda_{n,p}^{n-1},\lambda_{n,p}^{n-2},\ldots,\lambda_{n,p},1)$ is an eigenvector of $M_{n,p}$ to the eigenvalue $\lambda_{n,p}$. Since $\lambda_{n,p}>0$, this must be the unique PF eigenvector (up to normalisation). Normalising, we then get 
\[
\boldsymbol{R}=\frac{1}{\sum^{n-1}_{r=0}\lambda^r_{n,p}}(\lambda_{n,p}^{n-1},\lambda_{n,p}^{n-2},\ldots,\lambda_{n,p},1)^{T}.
\]
We briefly mention that $\boldsymbol{R}_i$ is the letter frequency of $a^{ }_i$ for any element $x\in X_{n,p}$, both in the deterministic and the random setting. 
For $m=1$, one has 
\[
\boldsymbol{q}_{1}=(\log(p+1),\log(p+1),\ldots,\log(p+1),0)
\]
since every letter $a^{ }_i$ has $p+1$ possible images under $\psi_{n,p}$ except for $a^{ }_n$, which only has $a^{ }_1$. Theorem~\ref{thm: estimate top entropy} implies the following bounds for the entropy in terms of $\lambda_{n,p}$. 

\begin{proposition}\label{prop: entropy bounds random noble pisa}
Let $\psi_{n,p}$ be a random noble Pisa substitution and $X_{n,p}$ the associated subshift. Then the topological entropy $h_{\textnormal{top}}(X_{n,p})$ admits the bounds given by 
\begin{equation}\label{eq: entropy bound lambda}
\log(p+1)\frac{\lambda^{n-1}_{n,p}-1}{\lambda^{n}_{n,p}-1}<h_{\textnormal{top}}(X_{n,p})<
\log(p+1)\frac{\lambda_{n,p}}{\lambda_{n,p}-1}\frac{\lambda^{n-1}_{n,p}-1}{\lambda^{n}_{n,p}-1},
\end{equation}
for all $n$ and $p$. \qed
\end{proposition}

Entropies of 
topologically mixing shifts of finite type (which can be recoded as random substitutions) are always logarithms of Perron (and hence algebraic) numbers. 
In \cite{GRS}, the authors asked whether it is true that all topological entropies of random substitution subshifts are logarithms of Perron numbers. 
Here, we provide a counterexample to this and provide sufficient conditions for which this holds. 

\begin{example}\label{ex: transcendence}
Let $\vartheta\colon a\mapsto {baa}, b\mapsto \left\{ab,ba\right\}$, which is a random variant of the squared Fibonacci substitution. 
Note that $\psi^{2}_{2,1}\colon a\mapsto \left\{aba,baa,aab\right\}, b\mapsto \left\{ab,ba\right\}$ and so $\vartheta$ can be obtained by removing $aba$ and $aab$ in the set of possible realisations in $\psi^{2}_{2,1}(a)$. 
 Its topological entropy is computed to be $h_{\text{top}}(X_{\vartheta})=\frac{1}{\tau^3}\log(2)$, where $\tau $ is the golden ratio. 
This was first computed by Nilsson in \cite{Nilsson} and was reconfirmed in \cite{Gohlke} via the estimates in Theorem~\ref{thm: estimate top entropy}. This random substitution satisfies the disjoint set condition, which means the lower bound for $m=1$ in Eq.~\eqref{eq: entropy bound gen} is the entropy itself. This number is a logarithm of the transcendental number $2^{\frac{1}{\tau^3}}$, where transcendence follows from the Gelfond--Schneider theorem since $2$ and $\frac{1}{\tau^3}$ are both algebraic and the latter is irrational. 
\end{example}

\begin{remark}
Using the same argument, it is easy to see that the bounds for $m=1$ in Eq.~\eqref{eq: entropy bound lambda} are all logarithms of transcendental numbers for any $n$ and $p$. However, since they do not satisfy both the disjoint set and the identical set condition, we do not resolve 
the transcendence question for this family. 
\end{remark}

We generalise the observation in Example~\ref{ex: transcendence} in the following result. 

\begin{proposition}\label{prop: transcendence}
Let $\vartheta$ be a primitive compatible random substitution over $\mathcal{A}=\left\{a^{ }_1,\ldots,a^{ }_n\right\}$ with irrational PF eigenvalue $\lambda$.
Let $\boldsymbol{R}_i$ be the letter frequency of $a_i$. 
Suppose the following conditions are 
satisfied: 
\begin{enumerate}
\item $\vartheta$ satisfies the identical set (disjoint set) condition for some $m\in\mathbb{N}$. 
\item The set $\left\{1,\lambda^{-m}\boldsymbol{R}_1,\ldots,\lambda^{-m}\boldsymbol{R}_n\right\}$ ( resp. $\left\{1,(\lambda^{m}-1)^{-1}\boldsymbol{R}_1,\ldots,(\lambda^{m}-1)^{-1}\boldsymbol{R}_n\right\}$) is linearly independent over $\mathbb{Q}$. 
\end{enumerate}
Then $h_{\text{top}}(X_{\vartheta})$ is a logarithm of a transcendental number. 
\end{proposition}

\begin{proof}
The entries of $\boldsymbol{q}_m$ are all logarithm of natural numbers. 
The entries of $\boldsymbol{R}$ are letter frequencies, which are all algebraic numbers (in fact lying in the same $\mathbb{Z}[\frac{1}{\lambda}]$-module).
Suppose $\vartheta$ satisfies the identical set condition for some $m$.
 The lower bound for $h_{\text{top}}(X_{\vartheta})$, which  in this case is the actual value of the entropy, is  of the form 
$\log(k^{\beta_1}_1\cdots k^{\beta_n}_n)$, where $k_i\in\mathbb{N}$, and $\beta_i=\frac{1}{\lambda^m}\boldsymbol{R}_i$ is algebraic for all $i$. We know that $\left\{1,\beta_1,\ldots,\beta_n\right\}$ are rationally independent from the second assumption. The transcendence of $\theta$ in $h_{\text{top}}(X_{\vartheta})=\log(\theta)$ the follows from Baker's theorem, which states that $k^{\beta_1}_1\cdots k^{\beta_n}_n$ is transcendental whenever $k_1,\ldots, k_n,\beta_1,\ldots, \beta_n$ are all algebraic and 
$\left\{1,\beta_1,\ldots,\beta_n\right\}$ is linearly independent over $\mathbb{Q}$ \cite{Baker}. The proof for when $\vartheta$ satisfies the disjoint set condition works the same way. 
\end{proof}

\begin{remark}
It is well known that
whenever $M$ has an irreducible characteristic polynomial, the entries
$\boldsymbol{R}_i$ of the right PF eigenvector are rationally independent. It is tempting to claim that irreducibility implies the second condition in Proposition~\ref{prop: transcendence}. However, this might not be true since there are cases when $\left\{\beta_1,\ldots,\beta_n\right\}$ is a rationally independent set but adding $1$ forces them to be dependent. As an example $\sum_{i=1}^{n} \boldsymbol{R}_i=1$ since $\boldsymbol{R}$ is a frequency vector, so $\left\{1,\boldsymbol{R}_1,\ldots,\boldsymbol{R}_n\right\}$ is a rationally dependent set. We still think that we can get the stronger linear independence property from irreducibility since we are dividing by $\lambda^{m}$ or $\lambda^{m}-1$, both of which are irrational, but so far we do not have a proof of this assertion. \exend
\end{remark}

\begin{question}\label{ques: transcendental}
Let $\vartheta$ be a primitive compatible random substitution whose PF eigenvalue $\lambda$ is irrational. 
Is it always true that $h_{\text{top}}(X_{\vartheta})=\log(\theta)$, where $\theta$ is a transcendental number?
\end{question}

Next, we derive explicit bounds for $\lambda_{n,p}$ solely in terms of $p$. 
Inserting $p$ and $p+1$, one gets $\chi_{n,p}(p)=p-\sum_{r=0}^{n-1}p^r<0$ and $\chi_{n,p}(p+1)=p$. This means $\chi_{n,p}(x)$ must have a real zero between $p$ and $p+1$. Since $p\geqslant 1$, and $\lambda_{n,p}$ is the unique root of modulus greater than 1, this zero must be $\lambda_{n,p}$. 
We then get that $p<\lambda_{n,p}<p+1$. Consequently, for a fixed $n$, the sequence $\left\{\lambda_{n,p}\right\}_{p\geqslant 1}$
is strictly increasing. The same estimate for a similar class of polynomials is found in \cite{GilWorley}. Using this estimate, one gets the following bounds which depend only on $n$ and $p$; see Figure~\ref{fig: entropy bounds} and Table~\ref{tab: entropy}. 

\begin{corollary}\label{coro: entropy bounds n,p}
Let $h_{\textnormal{top}}(X_{n,p})$ be as in Proposition~\textnormal{\ref{prop: entropy bounds random noble pisa}}. Then one has 
\begin{equation}\label{eq: entropy bound n p}
\log(p+1)\Big(\frac{p^{n-1}-1}{(p+1)^{n}-1}\Big)<h_{\textnormal{top}}(X_{n,p})<
\log(p+1)\Big(\frac{p+1}{p-1}\Big)\Big(\frac{(p+1)^{n-1}-1}{p^{n}-1}\Big)
\end{equation}
for $p>1$. \qed
\end{corollary}

\begin{remark}
In this work, we considered random substitutions as set-valued maps. One can equip each possible image of $a^{  }_i$ in Eq.~\eqref{eq: RNP definition} with a probability distribution $\nu^{ }_i\in\mathcal{M}(\psi(a^{ }_i))$, which yields a measure-theoretic dynamical system
$(X_{n,p},S,\mu)$, where the invariant measure $\mu$ depends on the probabilities; see \cite{GS}. For this system, one can compute the metric entropy $h_{\mu}(X_{n,p})$, which admits similar bounds to that given in Eq.~\eqref{eq: entropy bound gen}, where $\boldsymbol{q}_m$ now depends on $\left\{\nu^{ }_i\right\}$; see \cite{GMRS}.  \exend
\end{remark}

\begin{figure}[h!]
  \subfloat[Estimates via Eq.~\eqref{eq: entropy bound n p}]{
    \includegraphics[scale=0.32]{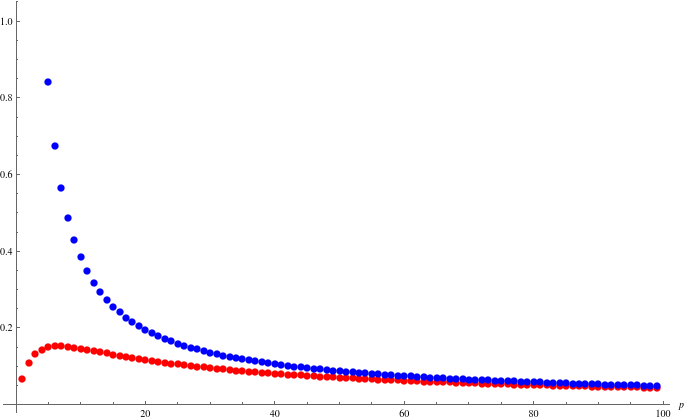}}
    \subfloat[Estimates via Eq.~\eqref{eq: entropy bound lambda}]{
    \includegraphics[scale=0.32]{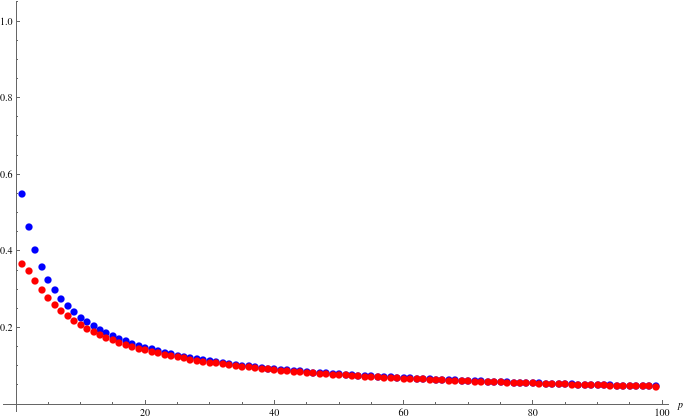}}
        \caption{Upper (blue) and lower (red) bounds for $h_{\text{top}}(X_{n,p})$ for $n=5$ and $2\leqslant p\leqslant 100$. Note that the bounds on the left look better for smaller values of $p$, but both are asymptotically equivalent as $p$ grows.  }
    \label{fig: entropy bounds}
\end{figure}

\begin{table}[h!]
    \centering
        \renewcommand{\arraystretch}{1.2}
    \begin{tabular}{|c|c c c c c c|}
    \hline 
      $p$  & $40$ & $41$ & $42$ & $43$ & $44$ & $45$  \\
      \hline
      $h^{+}_{5,p}$ & $0.107732$ & $0.105407$ & $0.103190$ & $0.101075$ & $0.099055$ &
       $0.097122$ \\ 
      $h^{-}_{5,p}$ & $0.082056$ & $0.080815$ & $0.079612$ & $0.078448$ & $0.077320$ & $0.076226$  
        \\
        \hline
    \end{tabular}
    \vspace*{2mm}
    
        \caption{Numerical values of the entropy bounds using Eq.~\eqref{eq: entropy bound n p} for $n=5$ and $40 \leqslant p\leqslant 45$. Here, $h^{+}_{n,p}$ and $h^{-}_{n,p}$ denote the upper and lower bounds, respectively.} 
    \label{tab: entropy}
\end{table}

The advantage of using Eq.~\eqref{eq: entropy bound n p} is that the bounds are given analytically in terms of $n$ and $p$, whereas one has to compute the eigenvalue $\lambda_{n,p}$ in order to use Eq.~\eqref{eq: entropy bound lambda}.

We then use the bounds we have in Proposition~\ref{prop: entropy bounds random noble pisa}
Corollary~\ref{coro: entropy bounds n,p} to demonstrate subshifts which are not topologically conjugate to a given random noble Pisa subshift.

Let $n=5$ and $p=3$. One can numerically plot the derivative of the upper bound of in Eq.~\eqref{eq: entropy bound n p} with respect to $p$, which reveals that the derivative is negative for $p>8$, thus implying that the upper bound is monotonically decreasing from $p=8$ onwards. Using Eq.~\eqref{eq: entropy bound n p}, we get  $h_{\text{top}}(X_{5,3})>0.10841:=h^{-}_{5,3}$. 
Since the  $h^{-}_{5,3}>h^{+}_{5,40}$ in Table~\ref{tab: entropy}, $X_{5,3}$ and $X_{5,40}$ necessarily have distinct entropies and cannot be topologically conjugate. In fact, if the upper bound is indeed monotonically decreasing for $p\geqslant 8$, one gets $X_{5,3}\not\simeq_{\text{top}} X_{5,p}$ for $p\geqslant 40$.

The finer estimate in Eq.~\eqref{eq: entropy bound lambda} yields a far better result, i.e., $X_{5,3}\not\simeq_{\text{top}} X_{5,p}$ already for $p\geqslant 5$. This is again under the assumption that the $\lambda_{5,p}$-dependent upper bound
is monotonically decreasing, which we do not prove here. Moll proved in \cite{Moll} that the $h_{\text{top}}(X_{2,p})$ for the random noble means family is actually strictly decreasing in $p$. The proof uses the closed form $\lambda_{2,p}=\frac{p+\sqrt{p^2+4}}{2}$, which we obviously do not have for general $n$.

\begin{remark}
It would be interesting to find an invariant which can distinguish whether $X_{n,p}$ and $X_{m,q}$ are topologically conjugate by just looking at the PF eigenvalue $\lambda_{n,p}$ and $\lambda_{m,q}$.  A result along this vein in the deterministic setting is Cobham's theorem, which says that if $\lambda$ and $\lambda^{\prime}$ are multiplicatively independent Perron numbers, a sequence $x\in\mathcal{A}^{\mathbb{N}}$ is both $\lambda$- and $\lambda^{\prime}$-substitutive if $x$ is ultimately periodic; see 
\cite{Durand}. Of course the obstacles one has to overcome in the random setting include the fact that $X_{\vartheta}$ is not linearly recurrent for a non-degenerate random substitution $\vartheta$ (although the set of linearly repetitive elements is dense in the subshift \cite[Prop.~16]{RustSpin}). \exend
\end{remark}

The following result describes the asymptotic behaviour of the sequence of entropies when one of the parameters is fixed and the other is sent to infinity. 

\begin{proposition}\label{prop: asymptotics}
Let $n$ be fixed. Then $\lim_{p\to \infty}h_{\textnormal{top}}(X_{n,p})=0$. 
On the other hand, for a fixed $p$, 
$\lim_{n\to\infty}h_{\textnormal{top}}(X_{n,p})=\frac{\log(p+1)}{(p+1)}$.
\end{proposition}

\begin{proof}
For the first claim, it suffices to show that both bounds in Eq.~\eqref{eq: entropy bound n p} converge to $0$ as $p\to\infty$, which we leave to the reader as an exercise; compare \cite[Sec.~3.2]{Moll}.
For the second claim, the upper and lower bounds in the coarser estimate involving only $n$ and $p$ converge to 
$0$ and $\infty$, respectively, which does not reveal anything about the asymptotics. 
This means one has to use the $\lambda_{n,p}$-dependent estimate in Eq.~\eqref{eq: entropy bound lambda}.
The crucial property here is $\lim_{p\to\infty}\lambda_{n,p}=p+1$, which we prove below. Let $\chi^{ }_{n,p}(x)$ be the characteristic polynomial above. For $0<\delta<1$, \[\chi^{ }_{n,p}(p+1-\delta)=\frac{p^2-\delta\big((p+1-\delta)^n-(p+1)\big)}{p-\delta},\]
which means one can choose an $N\in\mathbb{N}$ such that $\chi_{n,p}(p+1-\delta)<0$ for all $n\geqslant N$. Since $\chi_{n,p}(p+1)=p$ for all $n$, this implies that $\lambda_{n,p}\in (p+1-\delta,p+1)$ for all $n\geqslant N$, which implies the claim since $\delta$ is arbitrary. The claim for $\lim_{n\to\infty} h_{\textnormal{top}}(X_{n,p})$ then follows from standard arguments. 
\end{proof}

\begin{corollary}
For every $n\geqslant 2$ and $\delta>0$, there is a $p\in \mathbb{N}$ such that the topological entropy $h_{\textnormal{top}}(X_{n,p})$ of $X_{n,p}$ is less than $\delta$. \qed 
\end{corollary}

\begin{remark}
The previous result is reminiscent of a density-result for entropies of random substitutions in $\mathbb{R}_{\geqslant 0}$ in \cite[Thm.~6.8]{GRS}, where it was shown that there is a random substitution $\vartheta$ with entropy $(\frac{\ell}{m}+n)\log(2)$, for all $n,m,\ell\in\mathbb{N}$; see also \cite{KKL} for a construction of shift spaces with arbitrary entropy $t\in[0,1]$. \exend
\end{remark}

As mentioned above, for a fixed $p$, one can view the family $\left\{\psi^{ }_{n,p}\right\}_{n\geqslant 2}$ of random noble Pisa substitutions as finite approximants of the random $p$-infinibonacci substitution
\[
\psi_{\infty,p}\colon a^{ }_i\to \left\{a_1^{j}a^{ }_{i+1}a^{p-j}_{1}\mid 0\leqslant j\leqslant p\right\},\,a^{ }_{\infty}\to \left\{a_1^{j}a^{ }_{\infty}a^{p-j}_{1}\right\},
\] 
which can be thought of as a random substitution on a (compact) infinite alphabet (in this case given by the one-point compactification of $\mathbb{N}^{\ast}=\mathbb{N}\cup \left\{\infty\right\}$. This random substitution is constant length 
and primitive in some sense (note that the substitution matrix in this case is infinite-dimensional, and so the usual notion via powers of $M$ no longer works). To this random substitution, one can associate a subshift $X_{\infty,p}$ via the corresponding language, where one needs to include limits of sequences of $\psi^{ }_{\infty,p}$-legal words in the language. The asymptotics we get from Proposition~\ref{prop: asymptotics}
leads to the following question.

\begin{question}
Is the topological entropy of the random $p$-infinibonacci subshift equal to $\frac{\log(p+1)}{(p+1)}$? 
\end{question} 

\section{Conclusion}

In this work, we have shown that recognisable words exist for the noble Pisa family at all levels whenever $p>1$. 
It is not clear whether recognisable words exist outside the ones arising from the $\Gamma$-construction, and if they do, whether they also have a hierarchical structure, i.e., shorter recognisable words give rise to longer ones via appropriate realisations under the substitution.
 One can also ask whether the techniques used in this paper can be generalised to prove the same result for other families of substitutions. 
Some important features of the random substitutions which allow one to use the $\Gamma$-construction to build recognisable words are:
\begin{enumerate}
    \item The existence of a length hierarchy, i.e., a letter whose level-$n$ realisations are never shorter than the images of other letters. 
    \item The reflection invariance of the level-$n$ images $\vartheta^{n}(a)$.
\end{enumerate}
If another random substitution has these properties, we suspect that, under additional assumptions, one can extend the method used here to construct recognisable words. As mentioned above, we suspect that our results can be extended to some $\beta$-substitutions whenever $\beta$ is a simple Parry number. It would be interesting to see what happens when $\beta$ is no longer simple. In particular, this would imply that the random $\beta$-substitution associated to it is no longer irreducible (although we never used irreducibility in the proofs here). 

Another natural question is identifying how many length-$n$ words or level-$n$ inflation words are actually level-$m$ recognisable for a given random substitution $\vartheta$. If for some $N\in\mathbb{N}$, all legal words are level-$N$ recognisable, this implies that the subshift itself is recognisable; compare \cite{Miro}. It would also be of interest whether one can prove a quantitative relation between a notion of an ``asymptotic percentage" of recognisable inflation words, and the deviation of the actual value of the topological entropy $h_{\textnormal{top}}(X_{\vartheta})$ from the bounds in Theorem~\ref{thm: estimate top entropy}.

\section{Acknowledgements}
The authors would like to thank Michael Baake for suggestions and comments on the manuscript, and  Philipp Gohlke and Dan Rust for fruitful discussions.
GBE would like to acknowledge the support of Office of the Dean of the School of Science and Engineering of the Ateneo de Manila University, and the Department of Science and Technology of the Philippines through its Accelerated Science and Technology Human Resource Development Program (DOST-ASTHRDP).
EDM would like to acknowledge the support of the Alexander von Humboldt Foundation, Ateneo de Manila University and Bielefeld University, where a part of this work was completed. 
NM would like to acknowledge the support of the German Research Foundation (DFG) within the CRC1283.

\end{document}